\documentclass[11pt]{article}
\usepackage[margin=1in]{geometry}
\usepackage{amsmath,amssymb,amsthm}
\usepackage{upgreek}
\usepackage[T1]{fontenc}
\usepackage{lmodern}
\usepackage{hyperref}
\allowdisplaybreaks
\newtheorem{theorem}{Theorem}[section]
\newtheorem{lemma}[theorem]{Lemma}
\newtheorem{proposition}[theorem]{Proposition}
\theoremstyle{definition}
\newtheorem{definition}[theorem]{Definition}
\newtheorem{example}[theorem]{Example}

\theoremstyle{remark}
\newtheorem{remark}[theorem]{Remark}
\numberwithin{equation}{section}

\title{Operational Calculus for the $n$th-Level Prabhakar Type Fractional Derivative with Applications}
\author{Imtiaz Waheed$^{1,2}$, Erkinjon Karimov$^{2}$, Mujeeb ur Rehman$^{1}$\\[8pt]
\parbox{0.92\textwidth}{\centering\small
$^{1}$Department of Mathematics, School of Natural Sciences, National University of Sciences and Technology, Islamabad, Pakistan\\[3pt]
$^{2}$Ghent Analysis \& PDE Center, Department of Mathematics: Analysis, Logic and Discrete Mathematics, Ghent University, Krijgslaan 281, 9000 Ghent, Belgium\\[3pt]
Corresponding author: Imtiaz Waheed\\
\texttt{iwaheed.phdmath21sns@student.nust.edu.pk}}}
\date{}
\begin{document}
\maketitle
\begin{abstract}
This study investigates the $n$th-level Prabhakar fractional derivative, a generalization encompassing some well-known fractional derivatives. We establish fundamental properties, particularly the relationship with the corresponding Prabhakar fractional integral. Moreover, we develop Mikusi\'{n}ski-type operational calculus for this derivative, providing a framework for solving differential equations involving this fractional operator. Furthermore, we discuss the Laplace transform of the $n$th-level Prabhakar fractional derivative. To illustrate its application, we present analytical solutions of two problems: a fractional order ordinary differential equation and the time fractional heat equation, both of which include the $n$th-level Prabhakar derivative.
\medskip


\end{abstract}

\section{Introduction}
Fractional calculus (FC) originated from L’Hôpital’s classical query concerning the meaning of a derivative of non-integer order. In response, Leibniz philosophically predicted that such an idea would one day lead to significant applications \cite{oldham2010fractional}. Throughout the twentieth century, the mathematical properties and diverse applications of FC were extensively developed \cite{hilfer2000applications,samko1993fractional}. Fractional differential equations, in particular, have attracted increasing attention due to their successful applications in various fields such as economics \cite{scalas2000fractional,tarasov2020mathematical}, fractional kinetics \cite{sokolov2004fractional}, and electrodynamics \cite{tarasov2022general}.

Fractional calculus operators exist in numerous forms, many of which can be unified under generalized frameworks. Examples include general fractional calculus \cite{agrawal2012some}, $n$th level fractional derivatives \cite{luchko2020fractional}, modified Erdélyi–Kober and Prabhakar fractional calculus operators \cite{kiryakova1993generalized,prabhakar1971singular}. Several of these generalizations employ different kernel functions, such as analytic kernels \cite{luchko2021general} or Sonine kernels, which allow for more flexible modeling of nonlocal behaviors.

The origins of Prabhakar fractional calculus can be traced back to Prabhakar’s 1971 paper \cite{prabhakar1971singular}. However, later in 2004, Kilbas et al. \cite{kilbas2004generalized} introduced an operator that was only identified as a fractional integral. The concept of Prabhakar's fractional derivative was subsequently introduced in 2014 by Garra et al. \cite{garra2014hilfer}. Since then, the mathematical framework of Prabhakar fractional calculus has been further explored in numerous studies \cite{giusti2020practical,polito2015some}, with applications emerging in areas such as viscoelasticity and option pricing \cite{colombaro2018storage,tomovski2020applications}. This formulation can be viewed as a generalized class encompassing several other well-known operators in fractional calculus, including both singular and nonsingular types.

To avoid redundant research efforts, general fractional derivatives with Sonin kernels, which include Riemann–Liouville and Caputo-type derivatives, have been studied \cite{fernandez2019fractional}. Their unification as first-level derivatives and fundamental properties were established in \cite{luchko20221st}, extended to arbitrary orders in \cite{luchko20231st},  operational calculus for the first-level presented in \cite{al2022operational}, and generalized to $n$th-level derivatives in \cite{al2025fractional}, encompassing the Djrbashian–Nersessian derivatives.

The concept of operational calculus based on purely algebraic methods was first introduced in the 1950s by the Polish mathematician Jan Mikusiński ~\cite{mikusinski2014operational,yosida2012operational}. In his framework, the first-order derivative of a continuous function was interpreted as a multiplication operation within a special field of convolution quotients. The Mikusiński approach was later extended to specific cases of the hyper-Bessel differential operator \cite{ditkin1963theory} and subsequently to the general hyper-Bessel operator of order $n$ \cite{dimovski1966operational}.

During the 1990s, Mikusiński-type operational calculi were further developed for various fractional derivatives, including the multiple Erdély-Kober fractional derivative \cite{luchko1994operational}, the RL fractional derivative \cite{hadid1996operational,luchko1995exact}, and the Caputo fractional derivative \cite{luchko1999operational}. An operational calculus for the Hilfer fractional derivative was later proposed in \cite{hilfer2009operational}. Moreover, Mikusiński-type operational calculi have been developed for the generalized fractional derivatives (GFDs) \cite{luchko2022fractional} and the regularized GFDs \cite{al2022operational,luchko2021operational}. These frameworks have been successfully applied to obtain closed-form solutions for initial value problems involving multiterm fractional differential equations with sequential GFDs and sequential regularized GFDs, respectively \cite{al2025fractional,hilfer2009operational,luchko20221st,luchko20231st}.

Motivated by this analogy, we present the notion of the $n$th-level Prabhakar fractional derivative in this study. This generalized operator provides a unified framework in which results remain valid simultaneously for the RL, Caputo, and Hilfer derivatives, as well as for the $n$th-level general fractional derivative. The main objectives of the present paper are outlined below:

\begin{itemize}
\item[(i)] The $n$th-level Prabhakar fractional derivative is defined as a general operator that unifies and includes some well-known fractional derivatives as particular cases.
\item[(ii)]We establish a connection between the Prabhakar integral and the $n$th-level Prabhakar fractional derivative and derive fundamental properties, including its Laplace transform.
\item[(iii)] We develop the operational calculus structure associated with the $n$th-level Prabhakar fractional derivative, providing a framework for solving differential equations involving this operator.
\item[(iv)] Using the operational calculus technique, we obtain explicit solutions to a Cauchy-type initial value problem and a time-fractional heat equation that incorporates the $n$th-level Prabhakar fractional derivative.
\end{itemize}

The remainder of this paper is organized as follows. Section~\ref{S2} presents the necessary preliminaries and key mathematical results used throughout the paper. Section~\ref{S3} introduces the definition and fundamental properties of the proposed $n$th-level Prabhakar fractional derivative. The Laplace transform of the $n$th level Prabhakar fractional derivative is discussed in Section ~\ref{sLapl}. In Section~\ref{S4}, we develop the framework of the Mikusiński operational calculus for the $n$th-level Prabhakar fractional derivative. Finally, Section~\ref{S5} discusses some applications.

\section{Preliminaries}\label{S2}
\subsection{Fractional integrals and derivatives}
In this section, we present the fractional integrals and derivatives, collectively referred to as differintegrals, and some basic results.
\begin{definition}\textup{\cite{diethelm2010analysis,samko1993fractional}}
For $\upalpha \in \mathbb{C}$, with $\operatorname{Re}(\upalpha) > 0$, the RL fractional integral of order $\upalpha$ for a suitable function $f:(0,\infty)\to\mathbb{C}$ is given by:
\begin{equation}
{}_{\:0}^{\textup{R}} \mathbb{I}_x^\upalpha f(x)
=\frac{1}{\Gamma(\upalpha)} \int_0^x (x-t)^{\upalpha-1} f(t)~ \mathrm{d}t,
\quad x \in (0,\infty).
\end{equation}
\end{definition}
\begin{definition}\textup{\cite{caputo1967linear,dzherbashian2020fractional,samko1993fractional}}
Let $\upalpha \in \mathbb{C}$ with $\operatorname{Re}(\upalpha) \ge 0$. The RL and Caputo fractional derivatives of order $\upalpha$ are defined, respectively, as:
\begin{align}
{}_{\:0}^{\textup{R}} \mathbb{D}_x^\upalpha f(x)
&= \frac{\mathrm{d}^m}{\mathrm{d}x^m}
~{}_{\:0}^{\textup{R}} \mathbb{I}_x^{m-\upalpha} ~f(x),\\
{}_{\:0}^{\textup{c}} \mathbb{D}_x^\upalpha ~f(x)&= {}_{\:0}^{\textup{R}} \mathbb{I}_x^{m-\upalpha}
\frac{\mathrm{d}^m }{\mathrm{d}x^m}~f(x),
\end{align}
where $m = \lfloor \operatorname{Re}(\upalpha) \rfloor + 1 \in \mathbb{N},\, \operatorname{Re}\upalpha >0$, and $\upalpha \notin \mathbb{N}$.   
\end{definition}
Both definitions constitute natural extensions of the RL integral. Specifically, the RL derivative is obtained by analytical continuation of the RL integral to negative orders, following the convention
\begin{equation}\label{dertoint}
{}_{\:0}^\text{R} \mathbb{I}_x^{-\upalpha} = {}_{\:0}^\text{R} \mathbb{D}_x^\upalpha,
\end{equation}
that allows extension to all $\upalpha \in \mathbb{C}$.
In contrast, while RL fractional differential equations require fractional initial conditions, the Caputo derivative is often preferred in applications since it involves only classical integer-order initial conditions \cite{diethelm2010analysis}.

RL differintegrals have natural composition properties both when the inner operator is a fractional integral and when the outer operator is an ordinary derivative:
\begin{align}\label{rlsemigroup}
{}_{\:0}^{\textup{R}} \mathbb{I}_x^\upalpha \circ{ }_{\:0}^\textup{R} \mathbb{I}_x^\upbeta & ={ }_{\:0}^\textup{R} \mathbb{I}_x^{\upalpha+\upbeta}, \quad \operatorname{Re}(\upalpha), \operatorname{Re}(\upbeta) > 0,\\
\frac{\mathrm{d}^m}{\mathrm{~d} x^m} \circ { }_{\:0}^\textup{R} \mathbb{D}_x^\upalpha & = { }_{\:0}^\textup{R} \mathbb{D}_x^{m+\upalpha}, \quad m \in \mathbb{N}, \quad \upalpha \in \mathbb{C}.
\end{align}
We follow the convention \eqref{dertoint} so that the operator of order $\upalpha$ in each case may be either a fractional integral or a fractional derivative, according to the sign of $\upalpha$.

There is also a modified composition property for fractional derivatives, which includes a finite sum of initial value terms:
\begin{align}\label{inttermprop}
& { }_{\:0}^\textup{R} \mathbb{D}_x^\upalpha ~{ }_{\:0}^\textup{R} \mathbb{D}_x^\upbeta f(x) = { }_{\:0}^\textup{R}\mathbb{D}_x^{\upalpha+\upbeta} f(t)-\sum_{k=0}^{m-1} \frac{x^{-\upalpha-k-1}}{\Gamma(-\upalpha-k)} ~\left[{ }_{\:0}^\textup{R} \mathbb{D}_x^{\upbeta-k-1} f(x)\right]_{x=0},
\end{align}
where $m=\lfloor\operatorname{Re} (\upbeta)\rfloor+1$. For further details, we refer the reader to \textup{\cite{kilbas2006theory,samko1993fractional}}.
\begin{theorem}\label{FCPC}
For the Riemann--Liouville fractional derivatives ${ }_{\:0}^\textup{R} \mathbb{D}_x^\upalpha$, $\upalpha \geq 0$, and the Riemann--Liouville fractional integrals ${ }_{\:0}^\textup{R} \mathbb{I}_x^\upalpha$, $\upalpha \geq 0$, the relation
\begin{equation}\label{eq:fundamental_theorem_fc}
    { }_{\:0}^\textup{R} \mathbb{D}_x^\upalpha ~{ }_{\:0}^\textup{R} \mathbb{I}_x^\upalpha \varphi(x) = \varphi(x),
    \quad x \in [0,1]
\end{equation}
holds true on appropriate nontrivial spaces of functions.
\end{theorem}
Now, here we present the fractional integral and derivative known as the Prabhakar fractional integral. Before doing so, we first recall the definition of the Prabhakar function, also called the three-parameter Mittag-Leffler function, introduced by T. R. Prabhakar in 1971 \cite{prabhakar1971singular}, which is defined as
\begin{equation}\label{ml3}
E^{\gamma}_{\upalpha,\upbeta}(z) = \sum_{k=0}^{\infty} \frac{(\gamma)_k  z^k}{k!\,\Gamma(\upalpha k + \upbeta)}, 
\quad \upalpha, \upbeta, \gamma \in \mathbb{C},\ \operatorname{Re}(\upalpha) > 0,\ z \in \mathbb{C},
\end{equation}
where $(\gamma)_k \equiv \Gamma(\gamma + k)/\Gamma(\gamma)$ is the Pochhammer symbol and $\Gamma(\cdot)$ denotes the Euler gamma function. Note that $E^{\gamma}_{\upalpha,\upbeta}(z)$ is an entire function of order $\rho = 1/\operatorname{Re}(\upalpha)$.
\begin{definition}\textup{\cite{prabhakar1971singular}}
Let $f \in L^1[a, b], a < b$, then the  left- and right-sided Prabhakar integral operators are defined as
\begin{align}\label{pint}
&{}_{a}{\mathbb{E}}^{\upalpha,\upbeta,\gamma,\delta}_{x} f(x)
= \int_{a}^{x} (x - \xi)^{\upbeta - 1} \,
E^{\gamma}_{\upalpha,\upbeta}\!\left( \delta (x - \xi)^{\upalpha} \right) 
\, f(\xi)\, d\xi, 
\qquad x > a, \\
&{}_{x}{\mathbb{E}}^{\upalpha,\upbeta,\gamma,\delta}_{b} f(x)
= \int_{x}^{b} (\xi - x)^{\upbeta - 1} \,
E^{\gamma}_{\upalpha,\upbeta}\!\left( \delta (\xi - x )^{\upalpha} \right) 
\, f(\xi)\, d\xi, 
\qquad x < b. \label{rpfi}
\end{align}
where the condition $\upalpha > 0$ is necessary for the convergence of the series in \eqref{ml3}, whereas the restriction $\upbeta > 0$ is imposed in order to guarantee the convergence of the integral.
\end{definition}
This integral operator was originally proposed in \cite{prabhakar1971singular}, which is nowadays recognized as the Prabhakar fractional integral. Throughout this article, we denote the left-sided Prabhakar fractional integral by ${}_{a}{\mathbb{E}}_{x}^{\upalpha,\upbeta,\gamma,\delta}$, and the Prabhakar function (or the three-parameter Mittag-Leffler function) by $E^{\gamma}_{\upalpha,\upbeta}$.

The operator in \eqref{pint} reduces to the RL integral for $\gamma = 0$ or $\delta = 0$. The Prabhakar fractional derivatives of RL type, Caputo type, and Hilfer type are given by.
\begin{definition}\textup{\cite{garra2014hilfer,tomovski2020applications}}
For any $\upalpha, \upbeta, \gamma, \delta \in \mathbb{C}$ with $\text{Re}(\upalpha) > 0$
and $\text{Re}(\upbeta) \geq 0,a \geq 0$, the left-sided Prabhakar fractional derivatives of 
RL type, Caputo type, and Hilfer type are defined, 
respectively, as follows, each acting on a suitable function 
$f : (0,\infty) \to \mathbb{C}$:
\begin{align*}
&{}_{\:\:\; a}^{\textup{PR}}\mathbb{D}^{\upalpha,\upbeta,\gamma,\delta}_{x} f(x) 
= \mathbb{D}^m ~ {}_{a}\mathbb{E}^{\upalpha,m-\upbeta,-\gamma,\delta}_{x} ~f(x),\\
&{}_{\:\:\; a}^{\textup{PC}}\mathbb{D}^{\upalpha,\upbeta,\gamma,\delta}_{x} f(x) 
= {}_{a}\mathbb{E}^{\upalpha, m-\upbeta, -\gamma, \delta}_{x} ~ \mathbb{D}^m  ~f(x),\\
&{}_{\:\:\; a}^{\textup{PH}}\mathbb{D}^{\upalpha,\upbeta,\gamma,\delta;\theta}_{x} f(x)
= {}_{a}\mathbb{E}^{\upalpha,\theta(m-\upbeta),-\gamma\theta,\delta}_{x} ~\mathbb{D}^m ~
{}_{a}\mathbb{E}^{\upalpha,(1-\theta)(m-\upbeta),-\gamma(1-\theta),\delta}_{x}   
f(x),
\end{align*}
where $m = \lfloor \text{Re}(\upbeta) \rfloor + 1 \in \mathbb{N}$ in all three cases, and $0 \leq \theta \leq 1$ in the last case. Similarly, the right-sided Prabhakar fractional derivative can be defined by using the right-sided Prabhakar fractional integral given in equation \eqref{rpfi}. For further details on the left- and right-sided versions of these operators, we refer the reader to \textup{\cite{anastassiou2021foundations}}.
\end{definition}
\begin{lemma}\textup{\cite{kilbas2004generalized,prabhakar1971singular}}\label{semigrouprop}
The left-sided Prabhakar integral has a semigroup property in its second and third parameters:
\begin{equation*}
{}_{a}\mathbb{E}_x^{\upalpha, \beta_1, \gamma_1, \delta} \circ {}_{a}\mathbb{E}_x^{\upalpha, \beta_2, \gamma_2, \delta} =  {}_{a}\mathbb{E}_x^{\upalpha, \beta_1+\beta_2, \gamma_1+\gamma_2, \delta},   
\end{equation*}
where $\operatorname{Re} (\upalpha)>0$, $\operatorname{Re} (\upbeta_1) > 0$, and $\operatorname{Re} (\upbeta_2) > 0$. As a consequence of this, we have
\begin{equation*}
{}_{a}\mathbb{E}_x^{\upalpha, \upbeta, \gamma, \delta} \circ {}_{a}\mathbb{E}_x^{\upalpha, \upbeta, \gamma, \delta} \circ \cdots \circ {}_{a}\mathbb{E}_x^{\upalpha, \upbeta, \gamma, \delta}= {}_{a}\mathbb{E}_x^{\upalpha, m \upbeta, m \gamma, \delta},   
\end{equation*}
where $\operatorname{Re} (\upalpha) > 0, a \geq 0$ and $\operatorname{Re} (\upbeta) > 0$ and $m \in \mathbb{N}$.
\end{lemma}
Moreover, an additional noteworthy connection between the RL and Prabhakar integrals arises through their series representations. Specifically, for $f \in L^1[0, T]$, the parameters $\upalpha, \upbeta > 0$. 
\begin{lemma}\textup{\cite{fernandez2019series}}\label{Seriesformulae}
The left-sided Prabhakar integral and left-sided Prabhakar derivative of RL type have the following locally uniformly convergent series formulae:
\begin{align*}
&{}_{a}\mathbb{E}_x^{\upalpha, \upbeta, \gamma, \delta} f(x) =\sum_{k=0}^{\infty} \frac{(\gamma)_k ~\delta^k}{k!}~{ }_{\:a}^{\textup{R}} \mathbb{I}_x^{\upalpha k+\upbeta} f(x), \\
&{ }{ }_{\:\:\:a}^{\textup{PR}} \mathbb{D}_x^{\upalpha, \upbeta, \gamma, \delta} f(x) =\sum_{k=0}^{\infty} \frac{(-\gamma)_k~ \delta^k}{k!}~{ }_{\:a}^{\textup{R}}\mathbb{I}_x^{\upalpha k-\upbeta} f(x),
\end{align*}
in both cases functions $f:(0, \infty) \rightarrow \mathbb{C}$, $\upalpha, \upbeta, \gamma, \delta \in \mathbb{C}$, $a \geq 0$, $\operatorname{Re} \upbeta \geq 0$.
\end{lemma}
\subsection{Function Spaces}
In this subsection, we define the functional spaces that will be used in formulating our main results. The use of appropriately chosen spaces such as $ C_{\mu} $ allows the formulation of fractional differential and integral operators, ensuring that they are well defined, closed under the associated norms, and exhibit the required mapping properties within the operational calculus framework.
\begin{definition}\textup{\cite{diethelm2010analysis}}
Let $0 < \upalpha < 1$. The function space $X_0$ is defined by
\begin{equation*}
X_0 = \left\{\, f \in L^{1}(0,1) \;\Big|\;
{}{ }_{\:a}^{\mathrm{R}}\mathbb{I}_{x}^{1-\upalpha} f(x)
= {}{ }_{\:a}^{\mathrm{R}}\mathbb{I}_{x}^{1-\upalpha}\, {}{ }_{\:a}^{\mathrm{R}}\mathbb{I}_{x}^{\upalpha} \varphi(x)
= {}{ }_{\:a}^{\mathrm{R}}\mathbb{I}_{x}^{1} \varphi(x) , \varphi \in L^{1}(0,1)
\right\}.    
\end{equation*}
\end{definition}
\begin{definition}\textup{\cite{diethelm2010analysis}}
Let $AC([0,1])$ denote the space of absolutely continuous functions on $[0,1]$. A function $f$ belongs to this space if and only if there exists a function 
$\varphi \in L^{1}(0,1)$ such that
\begin{equation*}
f(x) = f(0) + \int_{0}^{x} \varphi(t)\,dt, 
\qquad x \in [0,1].    
\end{equation*}
That is,
\begin{equation*}
f \in AC([0,1])
\;\Longleftrightarrow\;
\exists\,\,\, \varphi \in L^{1}(0,1)
\ \text{such that}\ 
f(x) = f(0) + \int_{0}^{x} \varphi(t)\,dt.    
\end{equation*}
\end{definition}
For further details, we refer the reader to \textup{\cite{luchko2020fractional}}.
Next, we discuss the function space $C_{\mu}$ and its related subspaces, which are essential for the formulation of the Mikusiński structure.
\begin{definition}\textup{\cite{dimovski1966operational}}
For any $\mu \in \mathbb{R}$, the function space $C_{\mu}$ is defined as the set of all functions 
$f : (0, \infty) \to \mathbb{C}$ such that 
\begin{equation*}
f(x) = x^{p} f_{1}(x),
\qquad p > \mu, \quad f_{1} \in C[0, \infty).    
\end{equation*}
In particular, we are mainly interested in the case $\mu = -1$, since the space $C_{-1}$ is contained in $L_{Loc}^{1}(0, \infty)$ and includes $C[0, \infty)$.
\end{definition}
Next, we define the  subspaces, ensuring various differentiability conditions
required for defining fractional differential operators of several types:
\begin{definition}\textup{\cite{hadid1996operational,hilfer2009operational,luchko2021general}}
Let $\mu \ge -1$.
\begin{itemize}
    \item[\textup{(i)}] The space ${}^{\mathrm{R}}\Omega_{\mu}^{\upalpha}$, for a given $\upalpha > 0$, consists of all 
    functions $f \in C_{\mu}$ such that the left-sided RL derivatives 
    ${}_{\:a}^{\mathrm{R}}\mathbb{D}_{x}^{\upbeta} f$ exist as functions in $C_{\mu}$ for all 
    $\upbeta$ satisfying $0 \le \upbeta \le \upalpha$.

    \item[\textup{(ii)}] The space ${}^{\mathrm{H}}\Omega_{\mu}^{\upalpha}$, for a given $\upalpha > 0$, consists of all 
    functions $f \in C_{\mu}$ such that the left sided Hilfer derivatives 
    ${}_{\:a}^{\mathrm{H}}\mathbb{D}_{x}^{\upbeta, v} f$ exist as functions in $C_{\mu}$ for all 
    $\upbeta$ satisfying $0 \le \upbeta \le \upalpha$ and all $v \in [0,1]$.

    \item[\textup{(iii)}] The space $C_{\mu}^{n}$, for a given $n \in \mathbb{N}$, consists of all functions 
    $f \in C_{\mu}$ such that the $n$th-order derivative $f^{(n)}$ exists as a function in $C_{\mu}$.
\end{itemize}
All of these spaces are vector subspaces of $C_{\mu}$. Particular attention is given to the case $\mu = -1$.
\end{definition}
Hence, we retain the classical definitions of ${}^{\mathrm{R}}\Omega_{\mu}^{\upalpha}$ and ${}^{\mathrm{H}}\Omega_{\mu}^{\upalpha}$, with both $\mu$ and $\upalpha$ considered as real parameters.
\begin{proposition}\textup{\cite{dimovski1966operational,hadid1996operational,rani2022mikusinski}}
For any $\mu \ge -1, a \geq 0$, the function space $C_\mu$ forms a commutative ring under pointwise addition and Laplace convolution. Moreover, the left RL fractional integral and the left Prabhakar fractional integral act on this space as
\begin{equation*}
{}_{\:a}^{\textup{R}}\mathbb{I}_{x}^\upalpha : C_\mu \longrightarrow C_{\mu + \mathrm{Re}\,(\upalpha)} \subset C_\mu, 
\qquad
{}_a\mathbb{E}_x^{\upalpha, \upbeta, \gamma, \delta} : C_\mu \longrightarrow C_{\mu + \mathrm{Re}\,(\upbeta)} \subset C_\mu,
\end{equation*}
for any $\upalpha, \upbeta, \gamma, \delta \in \mathbb{C}$ with $\mathrm{Re}\,(\upalpha) > 0$ and $\mathrm{Re}\,(\upbeta) > 0$.
\end{proposition}
\begin{proposition}\textup{\cite{hadid1996operational,rani2022mikusinski}}
Every left-sided RL fractional derivative and every left-sided Prabhakar fractional derivative of RL type acts on the corresponding subspace of $C_{-1}$ as follows:
\begin{equation*}
{}_{\:a}^{\textup{R}} \mathbb{D}_x^\upalpha : {}^{\textup{R}} \Omega_{-1}^{\mathrm{Re}\,(\upalpha)} \longrightarrow C_{-1}, 
\qquad
{}_{\:\:\:a}^{\textup{PR}} \mathbb{D}_x^{\upalpha, \upbeta, \gamma, \delta} : {}^{\textup{R}} \Omega_{-1}^{\mathrm{Re}\,(\upbeta)} \longrightarrow C_{-1},
\end{equation*}
for any $\upalpha, \upbeta, \gamma, \delta \in \mathbb{C}$ with $\mathrm{Re}\,(\upalpha) \ge 0$ and $\mathrm{Re}\,(\upbeta) \ge 0$.
\end{proposition}
In the next section, we present the fractional derivative called the $n$th-level Prabhakar fractional derivative. 
\section{The \texorpdfstring{$n$}{nth}th-level Prabhakar fractional derivative}\label{S3}
The construction of the $n$th-level Prabhakar fractional derivative involves $n$ successive compositions of first-order derivatives and Prabhakar fractional integrals. Consequently, this process yields $n$ distinct families of parameter fractional derivatives, referred to as the $n$th-level Prabhakar fractional derivatives.

\begin{definition}
The domain of the $n$th-level Prabhakar fractional derivative can be extended to the following larger space of functions:
\begin{align*}
	&X_{nL} = \\
    &\Bigg\{ f  : h_i :=
	\left(\prod_{k=i}^{n} {}_{a}\mathbb{E}_{x}^{\upalpha,\, s_k,\, -\gamma\theta_k,\, \delta}\,\frac{d}{dx}\right)
	{}_{a}\mathbb{E}_{x}^{\upalpha,\, n-\upbeta-\upsigma_n,\, -\gamma(1-\Theta_n),\, \delta}\, f 
	 \in AC([a,b])
	\Bigg\}, 
\end{align*}
$h_{n+1}$ is innermost integral, \, $i=2,...,n+1$.
\end{definition}
Where we suppose that the parameters $s_1, s_2,\ldots,s_n \in \mathbb{R}$, and $\theta_1, \theta_2,\ldots,\theta_n \in \mathbb{R}$
satisfy the following conditions:
\begin{align}\label{peremeters}
&0< \upbeta\le 1,\, s_k\ge 0,\,\qquad \beta + \upsigma_k\le k\qquad (k=1,2,...,n)\nonumber\\
&n-1<\upbeta+\upsigma_n,\, \qquad s_k<1\,\qquad (k=2,...,n), \nonumber \\
&\text{and}\quad 0 \leq\theta_{k},\,\,\qquad \Theta_n\le 1\,\\
&\text{so that all}\,\,-\gamma\theta_k ~\text{and}\,\,-\gamma(1-\Theta_n)\,\,\text{have the same sign},\nonumber
\end{align}
where, for convenience, we used the notation
\begin{align*}
\upsigma_n = \sum_{i=1}^{n} s_i, \qquad\text{and} \quad
\Theta_{n} = \sum_{i=1}^{n} \theta_i, \quad i = 1, 2, \ldots, n. 
\end{align*}
Let $f \in X_{nL}, a < b$, then, the $n$th-level left Prabhakar fractional derivative is introduced as follows:
\begin{align}\label{basicdefn}
{}_{a}\mathbb{D}^{\upalpha,\upbeta, \gamma, \delta; \theta_{n}}_{nL, x} f(x) =& {}_{a}\mathbb{E}_{x}^{\upalpha, s_{1}, -\gamma \theta_{1}, \delta}\frac{\mathrm{~d}}{\mathrm{~d} x} {}_{a}\mathbb{E}_{x}^{\upalpha, s_{2}, -\gamma \theta_{2}, \delta} \frac{\mathrm{~d}}{\mathrm{~d} x} \cdots {}_{a}\mathbb{E}_{x}^{\upalpha, s_{n}, -\gamma \theta_{n}, \delta} \frac{\mathrm{~d}}{\mathrm{~d} x} \times \nonumber\\ 
&{}_{a}\mathbb{E}_{x}^{\upalpha,n - \upbeta - \upsigma_{n}, -\gamma(1- \Theta_{n}), \delta}  f(x), \nonumber\\
=& \left(\prod_{k=1}^{n} \left({}_{a}\mathbb{E}_{x}^{\upalpha, s_{k}, -\gamma \theta_{k}, \delta}\frac{\mathrm{~d}}{\mathrm{~d} x}\right)\right) {}_{a}\mathbb{E}_{x}^{\upalpha,n - \upbeta - \upsigma_{n}, -\gamma(1- \Theta_{n}), \delta}  f(x).
\end{align}
Note that for $s_k=0$
\begin{equation*}
{}_{a}\mathbb{E}_{x}^{\upalpha, 0, 0, \delta}= \mathrm{Id},\quad {}_{a}\mathbb{E}_{x}^{\upalpha, 0, \rho, \delta} = \mathrm{Id} + \sum\limits_{j\ge1} \frac{(\rho)_j \delta^j}{j!}~{}^{\textup{RL}}\mathbb{I}^{\upalpha j }\,\,\text{for}\,\, \rho\neq 0.    
\end{equation*}
Similarly, the $n$th-level right Prabhakar fractional derivative is introduced as follows: 
\begin{align}
{}_{x}\mathbb{D}^{\upalpha,\upbeta, \gamma, \delta; \theta_{n}}_{nL, b} f(x) =& {}_{x}\mathbb{E}_{b}^{\upalpha, s_{1}, -\gamma \theta_{1}, \delta}\left(-\frac{\mathrm{~d}}{\mathrm{~d} x}\right){}_{x}\mathbb{E}_{b}^{\upalpha, s_{2}, -\gamma \theta_{2}, \delta} \left(-\frac{\mathrm{~d}}{\mathrm{~d} x}\right) \cdots \times \nonumber\\
&{}_{x}\mathbb{E}_{b}^{\upalpha, s_{n}, -\gamma \theta_{n}, \delta} \left(-\frac{\mathrm{~d}}{\mathrm{~d} x}\right)  
{}_{x}\mathbb{E}_{b}^{\upalpha,n - \upbeta - \upsigma_{n}, -\gamma(1- \Theta_{n}), \delta}  f(x),\nonumber\\
=& \left(\prod_{k=1}^{n} \left({}_{x}\mathbb{E}_{b}^{\upalpha, s_{k}, -\gamma \theta_{k}, \delta}\left(-\frac{\mathrm{~d}}{\mathrm{~d} x}\right)\right)\right) {}_{x}\mathbb{E}_{b}^{\upalpha,n - \upbeta - \upsigma_{n}, -\gamma(1- \Theta_{n}), \delta}  f(x).
\end{align}
We can recover several well-known fractional derivatives as special cases of the general operator defined in~\eqref{basicdefn}, 
which we refer to as the left and right-sided $n$th-level Prabhakar fractional derivative.
\begin{itemize}
\item[(i)] In~\eqref{basicdefn}, by taking $n = 1$ and selecting $s_{1} = \theta_{1}(1 - \upbeta),$ the $n$th-level Prabhakar fractional derivative, it reduces to the Hilfer Prabhakar fractional derivative~\cite{garra2014hilfer,hilfer2000applications}.
\item[(ii)] By taking $n = 1$, $\theta_{1} = 1$, and choosing $s_{1} = \theta_{1}(1 - \upbeta),$ the $n$th level Prabhakar fractional derivative reduces to the Prabhakar Caputo fractional derivative~\cite{garra2014hilfer,hilfer2000applications}.

\item[(iii)] By taking $n = 1$, $\theta_{1} = 0$, and choosing $s_{1} = \theta_{1}(1 - \upbeta),$ the $n$th-level Prabhakar fractional derivative, it reduces to the Prabhakar RL fractional derivative~\cite{garra2014hilfer,hilfer2000applications}.

\item[(iv)] By setting $\delta = 0$ or $\gamma = 0$ in~\eqref{basicdefn}, the kernel of the $ n$th-level Prabhakar derivative reduces to a power function form, and the operator collapses to the $n$th-level fractional derivative introduced in~\cite{al2025fractional,luchko2020fractional}.

\item[(v)] The $n$th-level Prabhakar fractional derivative reduces to the classical Hilfer fractional derivative when $n=1$, $\gamma = 0$, and $s_{1} = \theta_{1}(1 - \upbeta);$ see~\cite{hilfer2000applications}.

\item[(vi)] For $n=1$, $s_1=0$, and $\gamma=0$ (equivalently \(\delta=0\) under the corresponding admissibility conditions), the proposed operator reduces to the classical Riemann–Liouville derivative of order $\beta$.~\cite{diethelm2010analysis}.

\item[(vii)] For $n=1$, $s_1=1-\beta$, and $\gamma=0$ (equivalently \(\delta=0\), under the corresponding assumptions), the proposed operator reduces to the classical Caputo derivative of order $\beta$.~\cite{caputo1967linear}. 
\end{itemize}

If we take $x=\psi(x)$ and define the parameters as in (i)--(vii), then the left and right sided versions of these operators are obtained; see~\cite{anastassiou2021foundations}.
These reductions demonstrate that the proposed $n$th-level Prabhakar fractional derivative 
serves as a unified framework that includes some classical fractional operators as particular cases.

Now, we consider a power function and determine its $n$th-level Prabhakar fractional derivative.
\begin{example}
Let $f(x) = x^r$, $r > -1$, and consider the $n=1$ in $n$th level Prabhakar fractional derivative defined by
\begin{equation*}
({}_{0}\mathbb{D}^{\upalpha,\upbeta,\gamma,\delta; \theta_{1}}_{1L, x} )f(x) 
=\left({}_{0}\mathbb{E}^{\upalpha,s_1,-\gamma \theta_{1},\delta}_{x} \frac{d}{dx}\right) 
\left({}_{0}\mathbb{E}^{\upalpha,1-\upbeta-s_1, -\gamma(1- \theta_{1})}_{x} f \right)(x).
\end{equation*}
Using the definition of the Prabhakar integral operator,
\begin{equation*}
{}_{0}\mathbb{E}^{\upalpha,\upbeta,\gamma,\delta}_x f(x) 
= \int_0^x (x-t)^{\upbeta-1} E_{\upalpha,\upbeta}^{\gamma}(\delta (x-t)^\upalpha) f(t)\, dt.
\end{equation*}
Applying the inner operator to $f(x) = x^r$ and using the result of \textup{\cite{prabhakar1971singular}},
\begin{align*}
{}_{0}\mathbb{E}^{\upalpha,1-\upbeta-s_1, -\gamma(1- \theta_{1}),\delta}_{x} ~x^r 
&= \int_0^x t^{r} (x-t)^{1-\upbeta-s_1-1} E_{\upalpha,1-\upbeta-s_1}^{-\gamma(1- \theta_{1})}(\delta (x-t)^\upalpha) \, dt \\
&= \Gamma(r+1)\,x^{r+1-\upbeta-s_1} E_{\upalpha,r+2-\upbeta-s_1}^{-\gamma(1- \theta_{1})}(\delta x^\upalpha).
\end{align*}
\begin{equation*}
\frac{d}{dx}~ {}_{0}\mathbb{E}^{\upalpha,1-\upbeta-s_1, -\gamma(1- \theta_{1}),\delta}_x (x^r)
= \Gamma(r+1)\,x^{r-\upbeta-s_1} E_{\upalpha,r+1-\upbeta-s_1}^{-\gamma(1- \theta_{1})}(\delta x^\upalpha).
\end{equation*}
Now applying ${}_{0}\mathbb{E}^{\upalpha,s_1, -\gamma(1- \theta_{1}) \theta_{1},\delta}$ yields
\begin{align}
{}_{0}\mathbb{E}^{\upalpha,s_1, -\gamma \theta_{1},\delta}_x &\left(\Gamma(r+1)~x^{r-\upbeta-s_1} E_{\upalpha,r+1-\upbeta-s_1}^{-\gamma(1- \theta_{1})}(\delta x^\upalpha)\right)\nonumber\\ 
=&\Gamma(r+1)\sum_{k=0}^{\infty}\frac{(-\gamma \theta_{1})_k\,\delta^k}{k!\,\Gamma(\upalpha k+s_1)} \times\nonumber\\
&\int_0^x (x-t)^{s_1-1+\upalpha k}t^{r-\upbeta-s_1}
E_{\upalpha,r+1-\upbeta-s_1}^{-\gamma(1- \theta_{1})}(\delta t^\upalpha)\,dt. \label{inteq}
\end{align}
Using the result of \textup{\cite{prabhakar1971singular}}, we have
\begin{align}\label{integralsol}
\int_0^x (x-t)^{s_1-1+\upalpha k}&t^{r-\upbeta-s_1}
E_{\upalpha,r+1-\upbeta-s_1}^{-\gamma(1- \theta_{1})}(\delta t^\upalpha)\,dt = \nonumber\\
& \Gamma(\upalpha k+s_1)\,x^{r-\upbeta+\upalpha k}
E_{\upalpha,r+1-\upbeta+\upalpha k}^{-\gamma(1- \theta_{1})}(\delta x^\upalpha).
\end{align}
Substituting \eqref{integralsol} into \eqref{inteq}, we obtain
\begin{align*}
&= \Gamma(r+1)
\sum_{k=0}^{\infty}\frac{(- \gamma \theta_{1})_k\,\delta^k}{k!}
x^{r-\upbeta+\upalpha k}
E_{\upalpha,r+1-\upbeta+\upalpha k}^{-\gamma(1- \theta_{1})}(\delta x^\upalpha).
\end{align*}
Here, $(- \gamma \theta_{1})_k$ represents the Pochhammer symbol. The first level Prabhakar fractional derivative of $x^r$ is
\begin{align*}
{}_{0}\mathbb{D}^{\upalpha,\upbeta,\gamma,\delta; \theta_{1}}_{1L, x}~x^r
=& \Gamma(r+1)
\sum_{k=0}^{\infty}\frac{(- \gamma \theta_{1})_k\,\delta^k}{k!} \times\\
&\quad x^{r-\upbeta+\upalpha k}
E_{\upalpha,r+1-\upbeta+\upalpha k}^{-\gamma(1- \theta_{1})}(\delta x^\upalpha).
\end{align*}
\end{example}
We now establish the Fundamental Theorem of FC (Theorem \ref{FCPC}, equation \eqref{eq:fundamental_theorem_fc}) for the $n$th-level Prabhakar fractional derivative. Throughout this section, we assume that $a<b$, $\upalpha>0$, and $\delta,\gamma\in\mathbb{R}$, while the parameters $\upbeta$, $s_k$, and $\theta_k$ satisfy the assumptions stated in the corresponding definition \ref{peremeters}. In particular, $0<\upbeta\leq 1$, $s_k\geq 0$, and $\upbeta+\upsigma_k\leq k$ for $k=1,\ldots,n$.

We shall use the following properties of the Prabhakar fractional integral on $L^1(a,b)$. First, the semigroup property holds as stated in Lemma \ref{semigrouprop}. Second, for every $\varphi\in L^1(a,b)$, the function ${}_{a}\mathbb{E}_{x}^{\upalpha,1,0,\delta}\varphi$ belongs to $AC([a,b])$, vanishes at $x=a$, and satisfies
\begin{equation}
\frac{d}{dx}\left({}_{a}\mathbb{E}_{x}^{\upalpha,1,0,\delta}\varphi\right)=\varphi(x)
\quad \text{for almost every } x\in(a,b).
\end{equation}
Furthermore, the operator ${}_{a}\mathbb{E}_{x}^{\upalpha,\upbeta,\gamma,\delta}$ is injective on $L^1(a,b)$ for every $\upbeta\geq 0$ and $\gamma \in \mathbb{R}$.

The function space $X_{FT}$, suitable for the $n$th-level Prabhakar fractional derivative, is defined as follows:
\begin{definition}
We define
\begin{equation}\label{eq:XFT}
X_{FT}=\left\{f\in L^1(a,b):\ {}_{a}\mathbb{E}_{x}^{\upalpha,\upbeta,\gamma,\delta}f\in AC([a,b])\ \text{and}\ {}_{a}\mathbb{E}_{x}^{\upalpha,\upbeta,\gamma,\delta}f(a)=0\right\}.
\end{equation}
\end{definition}
By the characterization of absolutely continuous functions, $f\in X_{FT}$ if and only if there exists $\varphi\in L^1(a,b)$ such that
\begin{equation}\label{eq:rep}
{}_{a}\mathbb{E}_{x}^{\upalpha,\upbeta,\gamma,\delta}f(x)
={}_{a}\mathbb{E}_{x}^{\upalpha,1,0,\delta}\varphi(x),
\qquad x\in[a,b].
\end{equation}
In this representation, $\varphi$ is uniquely determined almost everywhere and coincides with the derivative of ${}_{a}\mathbb{E}_{x}^{\upalpha,\upbeta,\gamma,\delta}f$ almost everywhere on $(a,b)$. 
\begin{theorem}\label{thm:FT}
On the space $X_{FT}$ the $n$th-level Prabhakar fractional derivative is a left-inverse operator to the Prabhakar fractional integral ${}_{a}\mathbb{E}_{x}^{\upalpha,\upbeta,\gamma,\delta}$, i.e., for every $f\in X_{FT}$ we have ${}_{a}\mathbb{E}_{x}^{\upalpha,\upbeta,\gamma,\delta}f\in X_{nL}$ and
\begin{equation}\label{eq:FT}
{}_{a}\mathbb{D}^{\upalpha,\upbeta,\gamma,\delta;\theta_{n}}_{nL,x}\,{}_{a}\mathbb{E}_{x}^{\upalpha,\upbeta,\gamma,\delta}f(x)=f(x),\qquad x\in[a,b].
\end{equation}
\end{theorem}
\begin{proof}
Let $f\in X_{FT}$. Substitute the representation \eqref{eq:rep} into the left-hand side of \eqref{eq:FT}. By Lemma \ref{semigrouprop}, 
\begin{align*}
{}_{a}\mathbb{E}_{x}^{\upalpha,n-\upbeta-\upsigma_n,-\gamma(1-\Theta_n),\delta}{}_{a}\mathbb{E}_{x}^{\upalpha,\upbeta,\gamma,\delta}f &={}_{a}\mathbb{E}_{x}^{\upalpha,n-\upbeta-\upsigma_n,-\gamma(1-\Theta_n),\delta}{}_{a}\mathbb{E}_{x}^{\upalpha,1,0,\delta}\varphi\\
&={}_{a}\mathbb{E}_{x}^{\upalpha,1,0,\delta}\,{}_{a}\mathbb{E}_{x}^{\upalpha,n-\upbeta-\upsigma_n,-\gamma(1-\Theta_n),\delta}\varphi .    
\end{align*}
We then obtain the following chain of relations:
\begin{align*}
&{}_{a}\mathbb{D}^{\upalpha,\upbeta,\gamma,\delta;\theta_{n}}_{nL,x}\,{}_{a}\mathbb{E}_{x}^{\upalpha,\upbeta,\gamma,\delta} f(x)\\ 
&=\prod_{k=1}^{n}\left({}_{a}\mathbb{E}_{x}^{\upalpha,s_k,-\gamma\theta_k,\delta}\frac{d}{dx}\right){}_{a}\mathbb{E}_{x}^{\upalpha,n-\upbeta-\upsigma_n,-\gamma(1-\Theta_n),\delta}{}_{a}\mathbb{E}_{x}^{\upalpha,\upbeta,\gamma,\delta}f(x)\\
&=\prod_{k=1}^{n}\left({}_{a}\mathbb{E}_{x}^{\upalpha,s_k,-\gamma\theta_k,\delta}\frac{d}{dx}\right) {}_{a}\mathbb{E}_{x}^{\upalpha,n-\upbeta-\upsigma_n,-\gamma(1-\Theta_n),\delta}{}_{a}\mathbb{E}_{x}^{\upalpha,1,0,\delta}\varphi(x)\\
&=\prod_{k=1}^{n-1}\left({}_{a}\mathbb{E}_{x}^{\upalpha,s_k,-\gamma\theta_k,\delta}\frac{d}{dx}\right){}_{a}\mathbb{E}_{x}^{\upalpha,s_n,-\gamma\theta_n,\delta}\frac{d}{dx}\,{}_{a}\mathbb{E}_{x}^{\upalpha,1,0,\delta}{}_{a}\mathbb{E}_{x}^{\upalpha,n-\upbeta-\upsigma_n,-\gamma(1-\Theta_n),\delta}\varphi(x)\\
&=\prod_{k=1}^{n-1}\left({}_{a}\mathbb{E}_{x}^{\upalpha,s_k,-\gamma\theta_k,\delta}\frac{d}{dx}\right){}_{a}\mathbb{E}_{x}^{\upalpha,s_n,-\gamma\theta_n,\delta}{}_{a}\mathbb{E}_{x}^{\upalpha,n-\upbeta-\upsigma_n,-\gamma(1-\Theta_n),\delta}\varphi(x)\\
&=\prod_{k=1}^{n-1}\left({}_{a}\mathbb{E}_{x}^{\upalpha,s_k,-\gamma\theta_k,\delta}\frac{d}{dx}\right){}_{a}\mathbb{E}_{x}^{\upalpha,n-\upbeta-\upsigma_{n-1},-\gamma(1-\Theta_{n-1}),\delta}\varphi(x)\\
&=\prod_{k=1}^{n-1}\left({}_{a}\mathbb{E}_{x}^{\upalpha,s_k,-\gamma\theta_k,\delta}\frac{d}{dx}\right){}_{a}\mathbb{E}_{x}^{\upalpha,n-1-\upbeta-\upsigma_{n-1},-\gamma(1-\Theta_{n-1}),\delta}{}_{a}\mathbb{E}_{x}^{\upalpha,1,0,\delta}\varphi(x)\\
&=\prod_{k=1}^{n-2}\left({}_{a}\mathbb{E}_{x}^{\upalpha,s_k,-\gamma\theta_k,\delta}\frac{d}{dx}\right){}_{a}\mathbb{E}_{x}^{\upalpha,s_{n-1},-\gamma\theta_{n-1},\delta}{}_{a}\mathbb{E}_{x}^{\upalpha,n-1-\upbeta-\upsigma_{n-1},-\gamma(1-\Theta_{n-1}),\delta}\varphi(x)\\
&~~~\;\vdots \\
&={}_{a}\mathbb{E}_{x}^{\upalpha,s_1,-\gamma\theta_1,\delta}\frac{d}{dx}\,{}_{a}\mathbb{E}_{x}^{\upalpha,1,0,\delta}\,{}_{a}\mathbb{E}_{x}^{\upalpha,1-\upbeta-s_1,-\gamma(1-\theta_1),\delta}\varphi(x)\\
&={}_{a}\mathbb{E}_{x}^{\upalpha,s_1,-\gamma\theta_1,\delta}\,{}_{a}\mathbb{E}_{x}^{\upalpha,1-\upbeta-s_1,-\gamma(1-\theta_1),\delta}\varphi(x)\\
&={}_{a}\mathbb{E}_{x}^{\upalpha,1-\upbeta,-\gamma,\delta}\varphi(x).
\end{align*}
Because $\varphi\in L^1(a,b)$, also ${}_{a}\mathbb{E}_{x}^{\upalpha,1-\upbeta,-\gamma,\delta}\varphi\in L^1(a,b)$, and we can apply the operator ${}_{a}\mathbb{E}_{x}^{\upalpha,\upbeta,\gamma,\delta}$ to the last formula of the chain. By Lemma \ref{semigrouprop} and \eqref{eq:rep},
\begin{align*}
{}_{a}\mathbb{E}_{x}^{\upalpha,\upbeta,\gamma,\delta} {}_{a}\mathbb{D}^{\upalpha,\upbeta,\gamma,\delta;\theta_{n}}_{nL,x}\,{}_{a}\mathbb{E}_{x}^{\upalpha,\upbeta,\gamma,\delta}f(x)
&={}_{a}\mathbb{E}_{x}^{\upalpha,\upbeta,\gamma,\delta}{}_{a}\mathbb{E}_{x}^{\upalpha,1-\upbeta,-\gamma,\delta}\varphi(x)\\
&={}_{a}\mathbb{E}_{x}^{\upalpha,1,0,\delta}\varphi(x)\\
&={}_{a}\mathbb{E}_{x}^{\upalpha,\upbeta,\gamma,\delta}f(x).
\end{align*}
The Prabhakar integral ${}_{a}\mathbb{E}_{x}^{\upalpha,\upbeta,\gamma,\delta}$ is injective on $L^1(a,b)$, and thus formula \eqref{eq:FT} follows.
\end{proof}
\begin{remark}
For $n=2$ the theorem is the fundamental theorem for the 2nd-level Prabhakar derivative, and for $\gamma=0$ (or $\delta=0$) it reduces to Theorem~3.4 of \cite{luchko2020fractional}, since ${}_{a}\mathbb{E}_{x}^{\upalpha,\mu,\rho,\delta}\to {}^{\textup{R}}\mathbb{I}^{\mu}$. For $n=1$ it covers the Hilfer--Prabhakar, the Prabhakar Riemann--Liouville and the Prabhakar Caputo derivatives.
\end{remark}
\begin{theorem}\label{thm:FT-right}
On the space $X^{\mathrm{r}}_{FT}$ the right $n$th-level Prabhakar fractional derivative is a left-inverse operator to the right Prabhakar fractional integral ${}_{x}\mathbb{E}_{b}^{\upalpha,\upbeta,\gamma,\delta}$, i.e., for every $f\in X^{\mathrm{r}}_{FT}$ we have ${}_{x}\mathbb{E}_{b}^{\upalpha,\upbeta,\gamma,\delta}f\in X^{\mathrm{r}}_{nL}$ and
\begin{equation}\label{eq:FT-right}
{}_{x}\mathbb{D}^{\upalpha,\upbeta,\gamma,\delta;\theta_{n}}_{nL,b}\,{}_{x}\mathbb{E}_{b}^{\upalpha,\upbeta,\gamma,\delta}f(x)=f(x),\qquad x\in[a,b].
\end{equation}
Here $X^{\mathrm{r}}_{FT}$ and $X^{\mathrm{r}}_{nL}$ are the spaces $X_{FT}$ and $X_{nL}$ with the left-sided operators replaced by the right-sided ones.
\end{theorem}
The proof follows along the same lines as that of Theorem~\ref{thm:FT}.

Fractional differential equations are often analyzed by reducing them to Volterra integral
equations, which requires acting on the equation with the corresponding fractional integral.
For the $n$th-level Prabhakar fractional derivative, this calls for an explicit expression of
${}_{a}\mathbb{E}_{x}^{\upalpha,\upbeta,\gamma,\delta}\,{}_{a}\mathbb{D}^{\upalpha,\upbeta,\gamma,\delta;\theta_{n}}_{nL,x}f$, the second fundamental theorem of calculus in this setting. We obtain it below by computing the
projector of this derivative.

For $k=1,\dots,n$ let $\upsigma_k=s_1+\dots+s_k$ and $\Theta_k=\theta_1+\dots+\theta_k$. The kernel functions of the $n$th-level Prabhakar derivative ${}_{a}\mathbb{D}^{\upalpha,\upbeta,\gamma,\delta;\theta_{n}}_{nL,x}$ are
\begin{equation}\label{def:kernel}
{}_{a}\mathcal{K}_{k}(x)=(x-a)^{\upbeta+\upsigma_k-k}\,E^{\gamma(1-\Theta_k)}_{\upalpha,\,\upbeta+\upsigma_k-k+1}\left(\delta(x-a)^{\upalpha}\right),\qquad k=1,\dots,n .
\end{equation}

Assume $n-1<\upbeta+\upsigma_n$ and $s_k<1$ for $k=2,\dots,n$. Then
\begin{equation}\label{prop:kernel}
\ker\,{}_{a}\mathbb{D}^{\upalpha,\upbeta,\gamma,\delta;\theta_{n}}_{nL,x}=\left\{\sum_{k=1}^{n}c_k\,{}_{a}\mathcal{K}_{k}(x)\ :\ c_1,\dots,c_n\in\mathbb{R}\right\},
\end{equation}
for $\gamma=0$ it reduces to the kernel of the $n$th-level Riemann--Liouville type derivative of \cite{luchko2020fractional}.

Under the conditions $\upbeta,s_k,\Theta_k$ stated in the equation \eqref{peremeters}, the projector operator for the $n$th-level Prabhakar fractional derivative is given by
\begin{theorem}\label{thm:proj}
Let $f$ belong to the domain of the $n$th-level Prabhakar
fractional derivative, and suppose that all the initial values
appearing below exist and that the stated compositions are
well defined. Then the operator
\begin{equation}\label{eq:12}
\mathrm{P}_{nL}^{\beta} =\mathrm{Id}- {}_{a}\mathbb{E}_{x}^{\alpha,\beta,\gamma,\delta}
  {}_{a}\mathbb{D}_{nL,x}^{\alpha,\beta,\gamma,\delta;
  {\theta}_{n}}
\end{equation}
admits the representation
\begin{align}\label{eq:13}
\mathrm{P}_{nL}^{\beta}f(x)=
\sum_{k=1}^{n} p_k
(x-a)^{\beta+\sigma_k-k}
E_{\alpha,\beta+\sigma_k-k+1}^{
\gamma(1-\Theta_k)}
\left(\delta(x-a)^\alpha\right),
\end{align}
for almost every $x\in(a,b)$, where
\begin{equation}\label{eq:14}
p_k= \left[\left(\prod_{i=k+1}^{n}{}_{a}\mathbb{E}_{x}^{\alpha,s_i,-\gamma\theta_i,\delta}
\frac{d}{dx}\right){}_{a}\mathbb{E}_{x}^{\alpha,n-\beta-\sigma_n,-\gamma(1-\Theta_n),\delta}f\right]_{x=a^+}
\end{equation}
Here, $k=1,\ldots,n$ and the product is taken as the identity operator
when $k=n$.
\end{theorem}
\begin{proof}
We start the proof of the theorem by introducing an auxiliary function.
\begin{equation}\label{eq:15}
g(x) ={}_{a}\mathbb{E}_{x}^{\upalpha,\upbeta,\gamma,\delta}\,{}_{a}\mathbb{D}^{\upalpha,\upbeta,\gamma,\delta;\theta_{n}}_{nL,x} f(x).
\end{equation}
For $f\in X_{nL}$, the derivative ${}_{a}\mathbb{D}^{\upalpha,\upbeta,\gamma,\delta;\theta_{n}}_{nL,x} f$ is well defined and belongs to $L^1(a,b)$. Apply ${}_{a}\mathbb{D}^{\upalpha,\upbeta,\gamma,\delta;\theta_{n}}_{nL,x}$ on $g(x)$, and using Theorem \ref{thm:FT}, we obtain
\begin{align*}
{}_{a}\mathbb{D}^{\upalpha,\upbeta,\gamma,\delta;\theta_{n}}_{nL,x}g(x)&= {}_{a}\mathbb{D}^{\upalpha,\upbeta,\gamma,\delta;\theta_{n}}_{nL,x}\,{}_{a}\mathbb{E}_{x}^{\upalpha,\upbeta,\gamma,\delta}\,{}_{a}\mathbb{D}^{\upalpha,\upbeta,\gamma,\delta;\theta_{n}}_{nL,x} f(x)\\
&={}_{a}\mathbb{D}^{\upalpha,\upbeta,\gamma,\delta;\theta_{n}}_{nL,x} f(x).
\end{align*}
This formula means that the function $g-f\in X_{nL}$ belongs to the kernel of the $n$th-level Prabhakar derivative. By equation \ref{prop:kernel} this kernel is spanned by the functions $(x-a)^{\upbeta+\upsigma_k-k}E^{\gamma(1-\Theta_k)}_{\upalpha,\upbeta+\upsigma_k-k+1}\big(\delta(x-a)^{\upalpha}\big)$, $k=1,\dots,n$. Thus, we obtain the representation.
\begin{equation}\label{eq:16}
g(x)=f(x)+\sum_{k=1}^{n}c_k\,(x-a)^{\upbeta+\upsigma_k-k}E^{\gamma(1-\Theta_k)}_{\upalpha,\upbeta+\upsigma_k-k+1}\left(\delta(x-a)^{\upalpha}\right).
\end{equation}
We now determine the coefficients $c_k$. For this, we apply a sequence of Prabhakar fractional integrals to the auxiliary function $g$, defined in \eqref{eq:15}. We employ the semigroup property presented in Lemma~\ref{semigrouprop}, the Prabhakar fractional integral of a power function, and the fundamental theorem of calculus \ref{thm:FT}. We further introduce the auxiliary functions
\begin{align*}
&h_i=\left(\prod_{k=i}^{n}{}_{a}\mathbb{E}_{x}^{\upalpha,s_k,-\gamma\theta_k,\delta}\frac{d}{dx}\right){}_{a}\mathbb{E}_{x}^{\upalpha,n-\upbeta-\upsigma_n,-\gamma(1-\Theta_n),\delta}f,\\
&h_{n+1}={}_{a}\mathbb{E}_{x}^{\upalpha,n-\upbeta-\upsigma_n,-\gamma(1-\Theta_n),\delta}f,
\end{align*}
and assume that $h_i\in AC([a,b])$ for $i=2,\dots,n+1$,
\begin{align*}
&~h_i={}_{a}\mathbb{E}_{x}^{\upalpha,s_i,-\gamma\theta_i,\delta}\frac{d}{dx}h_{i+1},\\
&{}_{a}\mathbb{D}^{\upalpha,\upbeta,\gamma,\delta;\theta_{n}}_{nL,x} f={}_{a}\mathbb{E}_{x}^{\upalpha,s_1,-\gamma\theta_1,\delta}\frac{d}{dx}h_2,
\end{align*}
and, by \eqref{eq:14}, $p_k=h_{k+1}(a)$, $k=1,\dots,n$. First, we apply the Prabhakar integral ${}_{a}\mathbb{E}_{x}^{\upalpha,1-\upbeta-s_1,-\gamma(1-\theta_1),\delta}$:
\begin{align}\label{h2}
{}_{a}\mathbb{E}_{x}^{\upalpha,1-\upbeta-s_1,-\gamma(1-\theta_1),\delta}\,&g(x) ={}_{a}\mathbb{E}_{x}^{\upalpha,1-\upbeta-s_1,-\gamma(1-\theta_1),\delta}\,{}_{a}\mathbb{E}_{x}^{\upalpha,\upbeta,\gamma,\delta}\,{}_{a}\mathbb{D}^{\upalpha,\upbeta,\gamma,\delta;\theta_{n}}_{nL,x} f(x)\nonumber\\
=&{}_{a}\mathbb{E}_{x}^{\upalpha,1-\upbeta-s_1,-\gamma(1-\theta_1),\delta}\,{}_{a}\mathbb{E}_{x}^{\upalpha,\upbeta,\gamma,\delta}\,{}_{a}\mathbb{E}_{x}^{\upalpha,s_1,-\gamma\theta_1,\delta}\frac{d}{dx}h_{2}(x)\nonumber\\
=&{}_{a}\mathbb{E}_{x}^{\upalpha,1-s_1,\gamma\theta_1,\delta}\,{}_{a}\mathbb{E}_{x}^{\upalpha,s_1,-\gamma\theta_1,\delta}\frac{d}{dx}h_{2}(x)\nonumber\\
=&{}_{a}\mathbb{E}_{x}^{\upalpha,1,0,\delta}\frac{d}{dx}h_{2}(x)\nonumber\\
{}_{a}\mathbb{E}_{x}^{\upalpha,1-\upbeta-s_1,-\gamma(1-\theta_1),\delta}\,=&~h_{2}(x)-h_{2}(a).
\end{align}
In the last equation, we used the action of the
Prabhakar fractional integral over a power function.
\begin{equation}\label{prul}
{}_{a}\mathbb{E}_{x}^{\upalpha,1-s_1,\gamma\theta_1,\delta} (x-a)^{0} = (x-a)^{1-s_1}E^{\gamma\theta_1}_{\upalpha,2-s_1}\left(\delta(x-a)^{\upalpha}\right).
\end{equation}
We also applied the composition property of the Prabhakar
fractional integral. Finally, the last equality follows from
the fundamental theorem of calculus, since
$h_2\in AC([a,b])$.

Again applying the Prabhakar integral ${}_{a}\mathbb{E}_{x}^{\upalpha,1-s_2,\gamma\theta_2,\delta}$ on equation \eqref{h2}:
\begin{align*}
{}_{a}\mathbb{E}_{x}^{\upalpha,1-s_2,\gamma\theta_2,\delta}\,&{}_{a}\mathbb{E}_{x}^{\upalpha,1-\upbeta-s_1,-\gamma(1-\theta_1),\delta}\,g(x) = {}_{a}\mathbb{E}_{x}^{\upalpha,1-s_2,\gamma\theta_2,\delta}\big(h_{2}-h_{2}(a)\big)(x)\\
&={}_{a}\mathbb{E}_{x}^{\upalpha,1,0,\delta}\frac{d}{dx}h_{3}(x)-h_{2}(a)\,{}_{a}\mathbb{E}_{x}^{\upalpha,1-s_2,\gamma\theta_2,\delta}1\\
&=h_{3}(x)-h_{3}(a)-h_{2}(a)\,(x-a)^{1-s_2}E^{\gamma\theta_2}_{\upalpha,\,2-s_2}\big(\delta(x-a)^{\upalpha}\big),
\end{align*}
where we used the identity
$h_2={}_{a}\mathbb{E}_{x}^{\upalpha,s_2,-\gamma\theta_2,\delta}
\frac{d}{dx}h_3$,
along with the semigroup property of the Prabhakar fractional
integral, the fundamental theorem of calculus, and the power rule.

By induction, we continue by applying the Prabhakar integrals ${}_{a}\mathbb{E}_{x}^{\upalpha,1-s_k,\gamma\theta_k,\delta}$, $k=3,\dots,n$, and thus arrive at the following formula: 
\begin{align}\label{eq:Eng}
&\left(\prod_{k=2}^{n}{}_{a}\mathbb{E}_{x}^{\upalpha,1-s_k,\gamma\theta_k,\delta}\right){}_{a}\mathbb{E}_{x}^{\upalpha,1-\upbeta-s_1,-\gamma(1-\theta_1),\delta}\,g(x) = {}_{a}\mathbb{E}_{x}^{\upalpha,n-\upbeta-\upsigma_n,-\gamma(1-\Theta_n),\delta}\,g(x)\nonumber\\
&\qquad={}_{a}\mathbb{E}_{x}^{\upalpha,n-\upbeta-\upsigma_n,-\gamma(1-\Theta_n),\delta}\,f(x) -\sum_{k=1}^{n} (x-a)^{m_{k}}E^{\gamma(\Theta_n-\Theta_k)}_{\upalpha,\,m_{k}+1}\left(\delta(x-a)^{\upalpha}\right)  \times \nonumber\\
&\qquad \quad \left[\left(\prod_{i=k+1}^{n}{}_{a}\mathbb{E}_{x}^{\upalpha,s_i,-\gamma\theta_i,\delta}\frac{d}{dx}\right) {}_{a}\mathbb{E}_{x}^{\upalpha,n-\upbeta-\upsigma_n,-\gamma(1-\Theta_n),\delta}\,f\right](a) ,
\end{align}
where $m_{k} = n-k-(\upsigma_n-\upsigma_k)$ and the empty product is understood as the identity operator. By the semigroup property, the composition of the applied operators is the Prabhakar integral with the order given below
\begin{align*}
&\quad(1-\upbeta-s_1)+\sum_{k=2}^{n}(1-s_k)=n-\upbeta-\upsigma_n, \\
&-\gamma(1-\theta_1)+\gamma(\theta_2+\dots+\theta_n)=-\gamma(1-\Theta_n),
\end{align*}
Now we apply the operator ${}_{a}\mathbb{E}_{x}^{\upalpha,n-\upbeta-\upsigma_n,-\gamma(1-\Theta_n),\delta}$ to equation \eqref{eq:16}. By linearity and the power rule given in equation \eqref{prul},
\begin{align}
&{}_{a}\mathbb{E}_{x}^{\upalpha,n-\upbeta-\upsigma_n,-\gamma(1-\Theta_n),\delta}\,g(x)
={}_{a}\mathbb{E}_{x}^{\upalpha,n-\upbeta-\upsigma_n,-\gamma(1-\Theta_n),\delta}\,f(x)\nonumber\\
&\qquad \quad +\sum_{k=1}^{n}c_k\,{}_{a}\mathbb{E}_{x}^{\upalpha,n-\upbeta-\upsigma_n,-\gamma(1-\Theta_n),\delta}\,(x-a)^{\upbeta+\upsigma_k-k}E^{\gamma(1-\Theta_k)}_{\upalpha,\upbeta+\upsigma_k-k+1}\left(\delta(x-a)^{\upalpha}\right),\nonumber\\
&\qquad ={}_{a}\mathbb{E}_{x}^{\upalpha,n-\upbeta-\upsigma_n,-\gamma(1-\Theta_n),\delta}\,f(x) +\sum_{k=1}^{n}c_k\,(x-a)^{m_k}E^{\gamma(\Theta_n-\Theta_k)}_{\upalpha,\,m_k+1}\left(\delta(x-a)^{\upalpha}\right) \label{eq:Eng2}.
\end{align}
where $n-\upbeta-\upsigma_n+\upbeta+\upsigma_k-k=m_k$. Thus, the left-hand side of \eqref{eq:Eng} and of \eqref{eq:Eng2} is the same function. The coefficients $c_k$, are obtained by comparison of the coefficients of the same functions $(x-a)^{m_k}E^{\gamma(\Theta_n-\Theta_k)}_{\upalpha,m_k+1}\big(\delta(x-a)^{\upalpha}\big)$ in the last two formulas \eqref{eq:Eng} and \eqref{eq:Eng2}. We obtain
\begin{equation*}
c_k=-\left[\left(\prod_{i=k+1}^{n}{}_{a}\mathbb{E}_{x}^{\upalpha,s_i,-\gamma\theta_i,\delta}\frac{d}{dx}\right){}_{a}\mathbb{E}_{x}^{\upalpha,n-\upbeta-\upsigma_n,-\gamma(1-\Theta_n),\delta}\,f\right](a)=-p_k.
\end{equation*}
The statement of the theorem follows from this formula and the representation \eqref{eq:16}:
\begin{align*}
\textup{P}^{\upbeta}_{nL} f(x)&=f(x)-g(x)\\
&=-\sum_{k=1}^{n}c_k\,(x-a)^{\upbeta+\upsigma_k-k}E^{\gamma(1-\Theta_k)}_{\upalpha,\upbeta+\upsigma_k-k+1}\big(\delta(x-a)^{\upalpha}\big)\\
&=\sum_{k=1}^{n}p_k\,(x-a)^{\upbeta+\upsigma_k-k}E^{\gamma(1-\Theta_k)}_{\upalpha,\upbeta+\upsigma_k-k+1}\big(\delta(x-a)^{\upalpha}\big).
\end{align*}
\end{proof}
\begin{remark}\label{rem:second}
Theorem~\ref{thm:proj} can be rewritten in the form
\begin{equation*}
{}_{a}\mathbb{E}_{x}^{\upalpha,\upbeta,\gamma,\delta}\,{}_{a}\mathbb{D}^{\upalpha,\upbeta,\gamma,\delta;\theta_{n}}_{nL,x} f(x)=f(x)-\sum_{k=1}^{n}p_k\,(x-a)^{\tau_k}E^{\gamma(1-\Theta_k)}_{\upalpha,\tau_k+1}\big(\delta(x-a)^{\upalpha}\big),
\end{equation*}
where the coefficients $p_k$, $k=1,\ldots,n$, depend on the function $f$ and are defined by \eqref{eq:14}. We refer to this result as the second fundamental theorem of the $n$th-level Prabhakar fractional derivative.
\end{remark}
\begin{remark}
For $n=2$, Theorem~\ref{thm:proj} reduces to Theorem~2 of \cite{luchko2020complete}, with $p_k\,x^{\tau_k}/\Gamma(\tau_k+1)$ in place of $p_k\,x^{\tau_k}$ where $\tau_k=\beta+\sigma_k-k$, and $p_k$ as in \eqref{eq:14} with Riemann--Liouville fractional integral instead of Prabhakar fractional integral. For $n=1$, they cover the Hilfer--Prabhakar, the Prabhakar Riemann--Liouville, and the Prabhakar Caputo derivatives.
\end{remark}

Next, we present an interpretation of the $n$th-level Prabhakar fractional derivative in the function space ${}^{\textup{H}}\Omega_{-1}^\upalpha$, and then establish the relationship between the Prabhakar integral and the $n$th-level Prabhakar fractional derivative.

\begin{theorem}\label{thm1111}
Let $n\in\mathbb{N}$, $\upalpha>0$, $\upbeta>0$, $\upsigma_n>0$ with $n-1<\upbeta+\upsigma_n\le n$, $\gamma,\delta\in\mathbb{R}$, and let $f\in{}^{\textup{H}}\Omega_{-1}^{\upbeta}$. Assume that the function ${}_{a}^{\mathrm{R}}\mathbb{I}_{x}^{\upalpha j+n-\upbeta-\upsigma_n}f$ belongs to $AC^{n}[a,b]$ for every $j\ge0$ and that every function to which $\frac{\mathrm{d}}{\mathrm{d}x}$ is applied in the definition of ${}_{a}\mathbb{D}^{\upalpha,\upbeta,\gamma,\delta;\theta_{n}}_{nL,x}f$ is absolutely continuous. Then
\begin{align}\label{eq:main}
{}_{a}\mathbb{D}^{\upalpha,\upbeta,\gamma,\delta;\theta_{n}}_{nL,x}f(x)
={}&\sum_{K=0}^{\infty}\frac{(-\gamma)_K\,\delta^{K}}{K!}\,{}_{a}^{\mathrm{R}}\mathbb{D}_{x}^{\upbeta-\upalpha K}f(x)\nonumber\\
&-\sum_{j=0}^{N}\frac{(-\gamma(1-\Theta_n))_j\,\delta^{j}}{j!}\,(x-a)^{\upsigma_n-n}E^{-\gamma\Theta_n}_{\upalpha,\,\upsigma_n-n+1}\left(\delta(x-a)^{\upalpha}\right) \times \nonumber\\
&\quad \left({}_{a}^{\mathrm{R}}\mathbb{I}_{x}^{\,n-\upbeta-\upsigma_n+\upalpha j}f\right)(a),
\end{align}
where
\begin{equation*}
N=\left\lfloor\frac{\upbeta+\upsigma_n-n+1}{\upalpha}\right\rfloor .
\end{equation*}
Moreover, ${}_{a}\mathbb{D}^{\upalpha,\upbeta,\gamma,\delta;\theta_{n}}_{nL,x}f\in C_{-1}$ if $\upsigma_n>n-1$ (in particular for $n=1$), or if $\left({}_{a}^{\mathrm{R}}\mathbb{I}_{x}^{\,n-\upbeta-\upsigma_n+\upalpha j}f\right)(a)=0$ for $j=0,\dots,N$.
\end{theorem}
\begin{proof}
By Lemma~\ref{Seriesformulae} applied to the $n+1$ Prabhakar operators in the definition,
\begin{align}\label{subseries}
{}_{a}\mathbb{D}^{\upalpha,\upbeta,\gamma,\delta;\theta_{n}}_{nL,x}f
&={}_{a}\mathbb{E}_{x}^{\upalpha,s_{1},-\gamma\theta_{1},\delta}\frac{\mathrm{d}}{\mathrm{d}x}\cdots{}_{a}\mathbb{E}_{x}^{\upalpha,s_{n},-\gamma\theta_{n},\delta}\frac{\mathrm{d}}{\mathrm{d}x}\,{}_{a}\mathbb{E}_{x}^{\upalpha,n-\upbeta-\upsigma_{n},-\gamma(1-\Theta_{n}),\delta}f(x)\nonumber\\
&=\sum_{i_1,\dots,i_n,j=0}^{\infty}\left(\prod_{r=1}^{n}\frac{(-\gamma\theta_r)_{i_r}\,\delta^{i_r}}{i_r!}\right)\frac{(-\gamma(1-\Theta_n))_j\,\delta^{j}}{j!}\nonumber\\
&\qquad\times{}_{a}^{\mathrm{R}}\mathbb{I}_{x}^{\nu_1}\frac{\mathrm{d}}{\mathrm{d}x}\,{}_{a}^{\mathrm{R}}\mathbb{I}_{x}^{\nu_2}\frac{\mathrm{d}}{\mathrm{d}x}\cdots{}_{a}^{\mathrm{R}}\mathbb{I}_{x}^{\nu_n}\frac{\mathrm{d}}{\mathrm{d}x}\,u_j(x),
\end{align}
where $\nu_r=\upalpha i_r+s_r$ for $r=1,\dots,n$, and
\begin{equation*}
u_j={}_{a}^{\mathrm{R}}\mathbb{I}_{x}^{\upalpha j+n-\upbeta-\upsigma_n}f .
\end{equation*}
The interchange of summation with the integrals and with $\frac{\mathrm{d}}{\mathrm{d}x}$ is justified by the term-by-term results of Luchko~\cite{luchko2020fractional}. We write
\begin{equation*}
S=\nu_1+\dots+\nu_n=\upalpha(i_1+\dots+i_n)+\upsigma_n,\qquad b_j=\upbeta+\upsigma_n-\upalpha j .
\end{equation*}
Let $\nu>0$ and $g\in AC[a,b]$. Then $g=g(a)+{}_{a}^{\mathrm{R}}\mathbb{I}_{x}^{1}g'$, so by the semigroup property \eqref{rlsemigroup}
\begin{align*}
{}_{a}^{\mathrm{R}}\mathbb{I}_{x}^{\nu}g(x)&=g(a)\frac{(x-a)^{\nu}}{\Gamma(\nu+1)}+{}_{a}^{\mathrm{R}}\mathbb{I}_{x}^{\nu+1}g'(x).
\end{align*}
This function is absolutely continuous, and differentiating gives
\begin{equation}\label{commutation}
\frac{\mathrm{d}}{\mathrm{d}x}\,{}_{a}^{\mathrm{R}}\mathbb{I}_{x}^{\nu}g(x)={}_{a}^{\mathrm{R}}\mathbb{I}_{x}^{\nu}g'(x)+g(a)\frac{(x-a)^{\nu-1}}{\Gamma(\nu)} .
\end{equation}
We claim that for $u\in AC^{n}[a,b]$ and $\nu_1,\dots,\nu_n\ge0$,
\begin{equation}\label{induction}
{}_{a}^{\mathrm{R}}\mathbb{I}_{x}^{\nu_1}\frac{\mathrm{d}}{\mathrm{d}x}\cdots{}_{a}^{\mathrm{R}}\mathbb{I}_{x}^{\nu_n}\frac{\mathrm{d}}{\mathrm{d}x}u={}_{a}^{\mathrm{R}}\mathbb{I}_{x}^{S}u^{(n)}+\sum_{k=0}^{n-2}u^{(n-1-k)}(a)\frac{(x-a)^{S-k-1}}{\Gamma(S-k)},
\end{equation}
where a term with $\Gamma(S-k)$ at a pole is zero. For $n=1$ this is the trivial identity ${}_{a}^{\mathrm{R}}\mathbb{I}_{x}^{\nu_1}u'={}_{a}^{\mathrm{R}}\mathbb{I}_{x}^{\nu_1}u'$. Suppose it holds for the $m$ innermost factors, with $S_m=\nu_{n-m+1}+\dots+\nu_n$ and $u^{(m)}$ in place of $u^{(n)}$. Applying $\frac{\mathrm{d}}{\mathrm{d}x}$ and using \eqref{commutation} for the first term and $\frac{\mathrm{d}}{\mathrm{d}x}\frac{(x-a)^{p}}{\Gamma(p+1)}=\frac{(x-a)^{p-1}}{\Gamma(p)}$ for the others,
\begin{align*}
\frac{\mathrm{d}}{\mathrm{d}x}\Bigg[{}_{a}^{\mathrm{R}}\mathbb{I}_{x}^{S_m}&u^{(m)}+\sum_{l=1}^{m-1}u^{(l)}(a)\frac{(x-a)^{S_m-m+l}}{\Gamma(S_m-m+l+1)}\Bigg]= \\
&{}_{a}^{\mathrm{R}}\mathbb{I}_{x}^{S_m}u^{(m+1)}+u^{(m)}(a)\frac{(x-a)^{S_m-1}}{\Gamma(S_m)} + \sum_{l=1}^{m-1}u^{(l)}(a)\frac{(x-a)^{S_m-m+l-1}}{\Gamma(S_m-m+l)}.
\end{align*}
Applying ${}_{a}^{\mathrm{R}}\mathbb{I}_{x}^{\nu_{n-m}}$ and using ${}_{a}^{\mathrm{R}}\mathbb{I}_{x}^{\nu}\frac{(x-a)^{q}}{\Gamma(q+1)}=\frac{(x-a)^{q+\nu}}{\Gamma(q+\nu+1)}$, we obtain \eqref{induction} for $m+1$ factors with $S_{m+1}=S_m+\nu_{n-m}$. This proves \eqref{induction}.

Since $u_j={}_{a}^{\mathrm{R}}\mathbb{I}_{x}^{n-b_j}f$, the definition of the Riemann--Liouville derivative gives, for $m=0,\dots,n$,
\begin{equation}\label{derivatives-of-u}
u_j^{(m)}={}_{a}^{\mathrm{R}}\mathbb{D}_{x}^{\,b_j-n+m}f,
\end{equation}
in particular,
\begin{equation}
u_j^{(n)}={}_{a}^{\mathrm{R}}\mathbb{D}_{x}^{\,b_j}f,\quad u_j^{(n-1-k)}(a)=\left({}_{a}^{\mathrm{R}}\mathbb{D}_{x}^{\,b_j-k-1}f\right)(a),
\end{equation}
where a derivative of negative order is an integral. Applying \eqref{induction} with $u=u_j$, the series \eqref{subseries} becomes
\begin{align}\label{series2}
{}_{a}\mathbb{D}^{\upalpha,\upbeta,\gamma,\delta;\theta_{n}}_{nL,x}f
&=\sum_{i_1,\dots,i_n,j=0}^{\infty}\left(\prod_{r=1}^{n}\frac{(-\gamma\theta_r)_{i_r}\,\delta^{i_r}}{i_r!}\right)\frac{(-\gamma(1-\Theta_n))_j\,\delta^{j}}{j!}\nonumber\\
&\times\left[{}_{a}^{\mathrm{R}}\mathbb{I}_{x}^{S}\,{}_{a}^{\mathrm{R}}\mathbb{D}_{x}^{b_j}f(x)+\sum_{k=0}^{n-2}\frac{(x-a)^{S-k-1}}{\Gamma(S-k)}\left({}_{a}^{\mathrm{R}}\mathbb{D}_{x}^{\,b_j-k-1}f\right)(a)\right].
\end{align}
The inner operator ${}_{a}^{\mathrm{R}}\mathbb{D}_{x}^{b_j}$ is a derivative for $b_j\ge0$, that is $j\le\left\lfloor(\upbeta+\upsigma_n)/\upalpha\right\rfloor$, and an integral for $b_j<0$. Let $M_j=\lfloor b_j\rfloor+1$ for $b_j\ge0$.

For $b_j<0$, \eqref{rlsemigroup} gives, since $b_j-S=\upbeta-\upalpha(i_1+\dots+i_n+j)$,
\begin{equation*}
{}_{a}^{\mathrm{R}}\mathbb{I}_{x}^{S}\,{}_{a}^{\mathrm{R}}\mathbb{D}_{x}^{b_j}f={}_{a}^{\mathrm{R}}\mathbb{D}_{x}^{\upbeta-\upalpha(i_1+\dots+i_n+j)}f .
\end{equation*}
For $b_j\ge0$, \eqref{inttermprop} with $m=M_j$ gives
\begin{equation*}
{}_{a}^{\mathrm{R}}\mathbb{I}_{x}^{S}\,{}_{a}^{\mathrm{R}}\mathbb{D}_{x}^{b_j}f={}_{a}^{\mathrm{R}}\mathbb{D}_{x}^{\upbeta-\upalpha(i_1+\dots+i_n+j)}f-\sum_{k=0}^{M_j-1}\frac{(x-a)^{S-k-1}}{\Gamma(S-k)}\left({}_{a}^{\mathrm{R}}\mathbb{D}_{x}^{\,b_j-k-1}f\right)(a).
\end{equation*}
Put $c_{j,k}=\left({}_{a}^{\mathrm{R}}\mathbb{D}_{x}^{\,b_j-k-1}f\right)(a)$. Inserting into \eqref{series2}, the term with $\frac{(x-a)^{S-k-1}}{\Gamma(S-k)}c_{j,k}$ appears with the sign $+$ for $0\le k\le n-2$ (from \eqref{induction}) and with the sign $-$ for $0\le k\le M_j-1$. Hence

\begin{itemize}
\item[(i)] for $0\le k\le\min(n-2,M_j-1)$ the two contributions cancel;
\item[(ii)] for $M_j\le k\le n-2$ (or all $k\le n-2$ if $b_j<0$) we have $k\ge b_j$, so $c_{j,k}=\left({}_{a}^{\mathrm{R}}\mathbb{I}_{x}^{\,k+1-b_j}f\right)(a)$ with $k+1-b_j>1$, and $c_{j,k}=0$ because ${}_{a}^{\mathrm{R}}\mathbb{I}_{x}^{1}f$ is continuous and vanishes at $a$;
\item[(iii)] for $k=n$ (possible only if $M_j=n+1$) we have $c_{j,n}=\left({}_{a}^{\mathrm{R}}\mathbb{I}_{x}^{1}f\right)(a)=0$.
\end{itemize}

Only the term $k=n-1$ survives, and it exists when $M_j\ge n$, that is $b_j\ge n-1$, that is $j\le N$ with $N=\left\lfloor\frac{\upbeta+\upsigma_n-n+1}{\upalpha}\right\rfloor$. Since $c_{j,n-1}=\left({}_{a}^{\mathrm{R}}\mathbb{D}_{x}^{\,b_j-n}f\right)(a)=\left({}_{a}^{\mathrm{R}}\mathbb{I}_{x}^{\,n-\upbeta-\upsigma_n+\upalpha j}f\right)(a)$, we obtain
\begin{equation*}
{}_{a}\mathbb{D}^{\upalpha,\upbeta,\gamma,\delta;\theta_n}_{nL,x}f=S_1-S_2,
\end{equation*}
\begin{align*}
S_1={}&\sum_{i_1,\ldots,i_n,j=0}^{\infty}\left(\prod_{r=1}^{n}\frac{(-\gamma\theta_r)_{i_r}\delta^{i_r}}{i_r!}\right)\frac{(-\gamma(1-\Theta_n))_j\delta^j}{j!}\,{}_{a}^{\mathrm{R}}\mathbb{D}_{x}^{\upbeta-\upalpha(i_1+\cdots+i_n+j)}f(x),\\
S_2={}&\sum_{j=0}^{N}\frac{(-\gamma(1-\Theta_n))_j\delta^j}{j!}\left({}_{a}^{\mathrm{R}}\mathbb{I}_{x}^{\,n-\upbeta-\upsigma_n+\upalpha j}f\right)(a)\\
&\quad\times\sum_{i_1,\ldots,i_n=0}^{\infty}\left(\prod_{r=1}^{n}\frac{(-\gamma\theta_r)_{i_r}\delta^{i_r}}{i_r!}\right)\frac{(x-a)^{\upalpha(i_1+\cdots+i_n)+\upsigma_n-n}}{\Gamma(\upalpha(i_1+\cdots+i_n)+\upsigma_n-n+1)}.
\end{align*}

We group the terms of $S_1$ with $i_1+\dots+i_n+j=K$ and use the Chu--Vandermonde identity
\begin{equation*}
\sum_{k_1+\dots+k_m=K}\ \prod_{r=1}^{m}\frac{(w_r)_{k_r}}{k_r!}=\frac{(w_1+\dots+w_m)_K}{K!},
\end{equation*}
with $m=n+1$, $w_r=-\gamma\theta_r$ for $r\le n$ and $w_{n+1}=-\gamma(1-\Theta_n)$, so that $w_1+\dots+w_{n+1}=-\gamma$:
\begin{align*}
S_1&=\sum_{K=0}^{\infty}\delta^{K}\sum_{i_1+\dots+i_n+j=K}\left(\prod_{r=1}^{n}\frac{(-\gamma\theta_r)_{i_r}}{i_r!}\right)\frac{(-\gamma(1-\Theta_n))_j}{j!}\;{}_{a}^{\mathrm{R}}\mathbb{D}_{x}^{\upbeta-\upalpha K}f(x)\\
&=\sum_{K=0}^{\infty}\frac{(-\gamma)_K\,\delta^{K}}{K!}\,{}_{a}^{\mathrm{R}}\mathbb{D}_{x}^{\upbeta-\upalpha K}f(x),
\end{align*}
which is the Prabhakar derivative of Riemann--Liouville type of $f$.

We group the terms of the sum over $i_1,\dots,i_n$ with $i_1+\dots+i_n=l$ and use the same identity with $m=n$, $w_1+\dots+w_n=-\gamma\Theta_n$:
\begin{align*}
\sum_{i_1,\dots,i_n=0}^{\infty}&\left(\prod_{r=1}^{n}\frac{(-\gamma\theta_r)_{i_r}\,\delta^{i_r}}{i_r!}\right)\frac{(x-a)^{\upalpha(i_1+\dots+i_n)+\upsigma_n-n}}{\Gamma(\upalpha(i_1+\dots+i_n)+\upsigma_n-n+1)}\\
&=\sum_{l=0}^{\infty}\delta^{l}\,\frac{(x-a)^{\upalpha l+\upsigma_n-n}}{\Gamma(\upalpha l+\upsigma_n-n+1)}\sum_{i_1+\dots+i_n=l}\ \prod_{r=1}^{n}\frac{(-\gamma\theta_r)_{i_r}}{i_r!}\\
&=(x-a)^{\upsigma_n-n}\sum_{l=0}^{\infty}\frac{(-\gamma\Theta_n)_l\left(\delta(x-a)^{\upalpha}\right)^{l}}{l!\,\Gamma(\upalpha l+\upsigma_n-n+1)}\\
&=(x-a)^{\upsigma_n-n}E^{-\gamma\Theta_n}_{\upalpha,\,\upsigma_n-n+1}\left(\delta(x-a)^{\upalpha}\right).
\end{align*}
Inserting this in $S_2$ gives \eqref{eq:main}.

$S_1$ is the Prabhakar derivative of Riemann--Liouville type of $f$, and belongs to $C_{-1}$ since $f\in{}^{\textup{H}}\Omega_{-1}^{\upbeta}\subset{}^{\textup{R}}\Omega_{-1}^{\upbeta}$. Each term of $S_2$ has the form constant $\times(x-a)^{\upsigma_n-n}\times$ an entire function of $(x-a)^{\upalpha}$. It lies in $C_{-1}$ if $\upsigma_n-n>-1$, that is $\upsigma_n>n-1$, which holds automatically for $n=1$, and otherwise if the constants $\left({}_{a}^{\mathrm{R}}\mathbb{I}_{x}^{\,n-\upbeta-\upsigma_n+\upalpha j}f\right)(a)$ vanish. Hence ${}_{a}\mathbb{D}^{\upalpha,\upbeta,\gamma,\delta;\theta_{n}}_{nL,x}f=S_1-S_2\in C_{-1}$ under the stated conditions.
\end{proof} 
\begin{remark}
Let $n=1$, $0<\upbeta<1$, $\nu\in(0,1]$, $s_1=\upsigma_1=\nu(1-\upbeta)$ and $\theta_1=\Theta_1=\nu$. Then $\upbeta+\upsigma_1=\upbeta+\nu(1-\upbeta)\le1$, so the conditions of Theorem~\ref{thm1111} hold, and $1-\upbeta-s_1=(1-\nu)(1-\upbeta)$. Hence ${}_{a}\mathbb{D}^{\upalpha,\upbeta,\gamma,\delta;\theta_{1}}_{1L,x}$ is the Hilfer--Prabhakar derivative
\begin{equation*}
{}_{a}\mathbb{D}^{\upalpha,\upbeta,\gamma,\delta;\nu}_{1L,x}f={}_{a}\mathbb{E}_{x}^{\upalpha,\nu(1-\upbeta),-\gamma\nu,\delta}\frac{\mathrm{d}}{\mathrm{d}x}\,{}_{a}\mathbb{E}_{x}^{\upalpha,(1-\nu)(1-\upbeta),-\gamma(1-\nu),\delta}f .
\end{equation*}
In \eqref{eq:main} we have $n=1$, $N=\left\lfloor(\upbeta+\upsigma_1)/\upalpha\right\rfloor$, and the boundary constants are $\left({}_{a}^{\mathrm{R}}\mathbb{I}_{x}^{(1-\nu)(1-\upbeta)+\upalpha j}f\right)(a)$. Assume that ${}_{a}^{\mathrm{R}}\mathbb{I}_{x}^{(1-\nu)(1-\upbeta)}f$ is continuous on $[a,b]$. For $j\ge1$,
\begin{equation*}
{}_{a}^{\mathrm{R}}\mathbb{I}_{x}^{(1-\nu)(1-\upbeta)+\upalpha j}f={}_{a}^{\mathrm{R}}\mathbb{I}_{x}^{\upalpha j}\,{}_{a}^{\mathrm{R}}\mathbb{I}_{x}^{(1-\nu)(1-\upbeta)}f
\end{equation*}
vanishes at $x=a$, so only the term with $j=0$ remains. Hence
\begin{equation*}
{}_{a}\mathbb{D}^{\upalpha,\upbeta,\gamma,\delta;\nu}_{1L,x}f(x)=\sum_{K=0}^{\infty}\frac{(-\gamma)_K\,\delta^{K}}{K!}\,{}_{a}^{\mathrm{R}}\mathbb{D}_{x}^{\upbeta-\upalpha K}f(x)-\mathrm{B}\,(x-a)^{\upsigma_1-1}E^{-\gamma\nu}_{\upalpha,\,\upsigma_1}\left(\delta(x-a)^{\upalpha}\right),
\end{equation*}
with $\upsigma_1=\nu(1-\upbeta)$ and $\mathrm{B}=\left({}_{a}^{\mathrm{R}}\mathbb{I}_{x}^{(1-\nu)(1-\upbeta)}f\right)(a)$. This is the Hilfer--Prabhakar relation of \cite[Theorem~3.5]{fernandez2023operational}.
\end{remark}
\begin{theorem}\label{thm:right}
Let $n\in\mathbb{N}$, $\upalpha>0$, $\upbeta>0$, $\upsigma_n>0$ with $n-1<\upbeta+\upsigma_n\le n$, $a<b$, $\gamma,\delta\in\mathbb{R}$, and let $f\in{}^{\textup{H}}\Omega_{-1}^{\upbeta}$. The right $n$th-level Prabhakar fractional derivative is
\begin{align*}
{}_{x}\mathbb{D}_{nL,b}^{\upalpha,\upbeta,\gamma,\delta;\theta_{n}}f= &{}_{x}\mathbb{E}_{b}^{\upalpha,s_{1},-\gamma\theta_{1},\delta}\left(-\frac{\mathrm{d}}{\mathrm{d}x}\right){}_{x}\mathbb{E}_{b}^{\upalpha,s_{2},-\gamma\theta_{2},\delta}\left(-\frac{\mathrm{d}}{\mathrm{d}x}\right)\cdots{}_{x}\mathbb{E}_{b}^{\upalpha,s_{n},-\gamma\theta_{n},\delta}\\
&\left(-\frac{\mathrm{d}}{\mathrm{d}x}\right){}_{x}\mathbb{E}_{b}^{\upalpha,n-\upbeta-\upsigma_{n},-\gamma(1-\Theta_{n}),\delta}f .
\end{align*}
Assume that ${}_{x}^{\mathrm{R}}\mathbb{I}_{b}^{\upalpha j+n-\upbeta-\upsigma_n}f\in AC^{n}[a,b]$ for every $j\ge0$ and that every function to which $-\frac{\mathrm{d}}{\mathrm{d}x}$ is applied in the definition is absolutely continuous. Then
\begin{align}\label{eq:right-main}
{}_{x}\mathbb{D}^{\upalpha,\upbeta,\gamma,\delta;\theta_{n}}_{nL,b}f(x)
={}&\sum_{K=0}^{\infty}\frac{(-\gamma)_K\,\delta^{K}}{K!}\,{}_{x}^{\mathrm{R}}\mathbb{D}_{b}^{\upbeta-\upalpha K}f(x)\nonumber\\
&-\sum_{j=0}^{N}\frac{(-\gamma(1-\Theta_n))_j\,\delta^{j}}{j!}\,(b-x)^{\upsigma_n-n}E^{-\gamma\Theta_n}_{\upalpha,\,\upsigma_n-n+1}\left(\delta(b-x)^{\upalpha}\right) \times \nonumber \\
&\left({}_{x}^{\mathrm{R}}\mathbb{I}_{b}^{\,n-\upbeta-\upsigma_n+\upalpha j}f\right)(b),
\end{align}
where
\begin{equation*}
N=\left\lfloor\frac{\upbeta+\upsigma_n-n+1}{\upalpha}\right\rfloor .
\end{equation*}
Moreover, ${}_{x}\mathbb{D}^{\upalpha,\upbeta,\gamma,\delta;\theta_{n}}_{nL,b}f\in C_{-1}$ if $\upsigma_n>n-1$ \textup{(}in particular for $n=1$\textup{)}, or if $\left({}_{x}^{\mathrm{R}}\mathbb{I}_{b}^{\,n-\upbeta-\upsigma_n+\upalpha j}f\right)(b)=0$ for $j=0,\dots,N$.
\end{theorem}
The proof follows along the same lines as that of Theorem~\ref{thm1111}. 
\section{Laplace transform of the \textit{n}th level Prabhakar fractional derivative}\label{sLapl}
In this section, we derive a formula for the Laplace transform of the $n$th level Prabhakar fractional derivative on the positive real semi-axis. We assume that the parameters satisfy the conditions \eqref{peremeters}.

In particular, for every $k=1,\dots,n$, and $ 
-1<\upbeta+\upsigma_k-k\le0$. Indeed, $\upbeta+\upsigma_k-k\le0$ follows from $\upbeta+\upsigma_k\le k$, and, since $s_i<1$ for $i\ge2$, we get $\upbeta+\upsigma_k>n-1-(\upsigma_n-\upsigma_k)>n-1-(n-k)=k-1$.

The Laplace transform of $f(t)$ is given by
\begin{equation*}
\mathcal{L}\left(f(t)\right)(\rho)=\int_0^{+\infty}f(t)e^{-\rho t}\,dt
\end{equation*}
exist. Since ${}_{0}\mathbb{E}_{x}^{\upalpha,\mu,\eta,\updelta}f$ is the Laplace convolution of $f$ with $x^{\mu-1}E^{\eta}_{\upalpha,\mu}(\updelta x^{\upalpha})$, the Laplace transform of the Prabhakar fractional integral is given by
\begin{align}
&\mathcal{L}\left({}_{0}\mathbb{E}_{x}^{\upalpha,\mu,\eta,\updelta}f\right)(\rho) = \rho^{-\mu}\left(1-\updelta \rho^{-\upalpha}\right)^{-\eta}\,\left(\mathcal{L}f\right)(\rho) , \label{L4}\\
&\mathcal{L}\left[x^{\upbeta+\upsigma_k-k}E^{\gamma(1-\Theta_k)}_{\upalpha,\upbeta+\upsigma_k-k+1}\left(\updelta x^{\upalpha}\right)\right](\rho) = \rho^{\,k-\upbeta-\upsigma_k-1}\left(1-\updelta \rho^{-\upalpha}\right)^{-\gamma(1-\Theta_k)},\label{L6}
\end{align}
where $\mu>0$, in \eqref{L4}, $\operatorname{Re}(\rho) > |\updelta|^{1/\upalpha}$ in \eqref{L6}, and $k=1,2,\dots,n$.

For the appropriate spaces of functions, the Fundamental Theorem of FC and the projector formula given in Theorem \ref{thm:proj} remain true for the operators ${}_{0}\mathbb{E}_{x}^{\upalpha,\upbeta,\gamma,\updelta}$ and ${}_{0}\mathbb{D}_{nL,x}^{\upalpha,\upbeta,\gamma,\updelta;\theta_n}$ defined on the positive real semi-axis. The reason is that these formulas are valid pointwise: for an arbitrary point $x\in\mathbb{R}_+$ we can choose an interval $[0,X]$ with $x\in[0,X]$, and on this interval the formulas hold. Thus, for $f\in X_{nL}$, Theorem \ref{thm:proj} gives the representation
\begin{equation}\label{L7}
{}_{0}\mathbb{E}_{x}^{\upalpha,\upbeta,\gamma,\updelta}\,{}_{0}\mathbb{D}_{nL,x}^{\upalpha,\upbeta,\gamma,\updelta;\theta_n}f(x)=f(x)-\sum_{k=1}^{n}p_k\,x^{\upbeta+\upsigma_k-k}E^{\gamma(1-\Theta_k)}_{\upalpha,\upbeta+\upsigma_k-k+1}\big(\updelta x^{\upalpha}\big),
\end{equation}
where the coefficients $p_k$, $k=1,\dots,n$, are given by
\begin{equation}\label{L8}
p_k = \left(\prod_{i=k+1}^{n}\Big({}_{0}\mathbb{E}_{x}^{\upalpha,s_i,-\gamma\theta_i,\updelta}\frac{d}{dx}\Big)\;{}_{0}\mathbb{E}_{x}^{\upalpha,n-\upbeta-\upsigma_n,-\gamma(1-\Theta_n),\updelta}f\right)(0).
\end{equation}
Let $f$ and ${}_{0}\mathbb{D}_{nL,x}^{\upalpha,\upbeta,\gamma,\updelta;\theta_n}f$ possess Laplace transforms. Applying the Laplace transform to \eqref{L7}, using \eqref{L6}, we get
\begin{align*}
\rho^{-\upbeta}
\left(1-\updelta \rho^{-\upalpha}\right)^{-\gamma}
\mathcal{L}\left(
{}_{0}\mathbb{D}_{nL,x}^{\upalpha,\upbeta,\gamma,\updelta;\theta_n}f\right)&(\rho) = \mathcal{L}(f)(\rho) - \\
&\sum_{k=1}^{n} p_k\, \rho^{\,k-\upbeta-\upsigma_k-1}
\left(1-\updelta \rho^{-\upalpha}\right)^{-\gamma(1-\Theta_k)}.
\end{align*}
Now we solve this equation for the Laplace transform of the $n$th level Prabhakar derivative. Multiplying by $\rho^{\upbeta}(1-\updelta \rho^{-\upalpha})^{\gamma}$ and using $\upbeta-(\upbeta+\upsigma_k-k)-1=k-\upsigma_k-1$ and $\gamma-\gamma(1-\Theta_k)=\gamma\Theta_k$, we arrive at the final formula
\begin{align}\label{L9}
\mathcal{L}\left({}_{0}\mathbb{D}_{nL,x}^{\upalpha,\upbeta,\gamma,\updelta;\theta_n}f\right)(\rho)=& ~\rho^{\upbeta}\left(1-\updelta \rho^{-\upalpha}\right)^{\gamma}\left(\mathcal{L}f\right)(\rho) \nonumber\\ 
&-\sum_{k=1}^{n}p_k\,\rho^{\,k-\upsigma_k-1}\left(1-\updelta \rho^{-\upalpha}\right)^{\gamma\Theta_k},
\end{align}
\noindent\textbf{Limiting cases.}
Let us consider the Laplace transform formula for the second-level Prabhakar fractional derivative ($n=2$):
\begin{align}\label{L10}
\mathcal{L}\left(
{}_{0}\mathbb{D}_{2L,x}^{\upalpha,\upbeta,\gamma,\updelta;\theta_2}
f(x)\right)(\rho) &= \rho^{\upbeta}
\left(1-\updelta \rho^{-\upalpha}\right)^{\gamma}
\mathcal{L}(f)(\rho) \nonumber\\
&\quad -p_1\,\rho^{-s_1} \left(1-\updelta \rho^{-\upalpha}\right)^{\gamma\theta_1} \nonumber\\
&\quad -p_2\,\rho^{1-s_1-s_2}
\left(1-\updelta \rho^{-\upalpha}\right)^{\gamma(\theta_1+\theta_2)}.
\end{align}
where $p_1$ and $p_2$ are given by
\begin{align*}
p_1&=\left({}_{0}\mathbb{E}_{x}^{\upalpha,s_2,-\gamma\theta_2,\updelta}\frac{d}{dx}\,{}_{0}\mathbb{E}_{x}^{\upalpha,2-\upbeta-s_1-s_2,-\gamma(1-\theta_1-\theta_2),\updelta}f\right)(0),\\
p_2&=\left({}_{0}\mathbb{E}_{x}^{\upalpha,2-\upbeta-s_1-s_2,-\gamma(1-\theta_1-\theta_2),\updelta}f\right)(0).
\end{align*}
For $\gamma=0$ (or $\updelta=0$) all Prabhakar integrals reduce to the Riemann--Liouville integrals of the same orders, all factors $(1-\updelta \rho^{-\upalpha})^{\cdot}$ become $1$, and \eqref{L9} turns into Luchko's formula for the Laplace transform of the $n$th level fractional derivative. For $n=1$, $\theta_1=\theta$ and $s_1=\theta(1-\upbeta)$, formula \eqref{L9} gives
\begin{multline}\label{L11}
\mathcal{L}\left(\,{}_{0}\mathbb{D}_{1L,x}^{\upalpha,\upbeta,\gamma,\updelta;\theta}f\right)(\rho)=\rho^{\upbeta}\left(1-\updelta \rho^{-\upalpha}\right)^{\gamma}\mathcal{L}(f)(\rho)\\-\rho^{-\theta(1-\upbeta)}\left(1-\updelta \rho^{-\upalpha}\right)^{\gamma\theta}\left({}_{0}\mathbb{E}_{x}^{\upalpha,(1-\theta)(1-\upbeta),-\gamma(1-\theta),\updelta}f\right)(0),
\end{multline}
which is the Laplace transform of the Hilfer--Prabhakar fractional derivative \cite{garra2014hilfer}. The case $\theta=0$ in \eqref{L11} corresponds to the Riemann--Liouville--Prabhakar derivative,
\begin{equation*}
\rho^{\upbeta}\left(1-\updelta \rho^{-\upalpha}\right)^{\gamma}\mathcal{L}(f)(\rho)-\left({}_{0}\mathbb{E}_{x}^{\upalpha,1-\upbeta,-\gamma,\updelta}f\right)(0),
\end{equation*}
whereas the case $\theta=1$ leads to the Laplace transform of the Caputo--Prabhakar derivative,
\begin{equation*}
\rho^{\upbeta}\left(1-\updelta \rho^{-\upalpha}\right)^{\gamma}\mathcal{L}(f)(\rho)-f(0)\,\rho^{\upbeta-1}\left(1-\updelta \rho^{-\upalpha}\right)^{\gamma}.
\end{equation*}
\section{Operational calculus for the \texorpdfstring{$n$}{nth}th-level Prabhakar fractional derivative}\label{S4}
In this section, we develop a Mikusiński operational calculus for the $n$th-level Prabhakar fractional derivative defined in \eqref{basicdefn}. The fundamental idea of Mikusiński-type operational calculus is to interpret differential, integral, or integro-differential operators as purely algebraic entities. Within this framework, equations involving such operators can be transformed into algebraic equations whose solutions represent generalized solutions of the corresponding differential, integral, or integro-differential equations. In some cases, by means of suitable operational relations, these generalized solutions can further be interpreted as strong (classical) solutions.

The starting point of the Mikusiński-type operational calculi for various classes of fractional operators is the observation that the set $\mathcal{R}_{-1} = \left(C_{-1}(0,+\infty), +, *\right),$ equipped with the usual addition $+$ and the Laplace convolution $*$, forms a commutative ring without zero divisors, as first established in \cite{mikusinski2014operational}.

As in the case of Mikusi\'nski's type operational calculus we have the following theorem:
\begin{theorem} \textup{\cite{rani2022mikusinski,mikusinski2014operational}}
The space $C_{-1}$ with the operations of the Laplace convolution $*$ and ordinary addition becomes a commutative ring $(C_{-1}, *, +)$ without divisors of zero.
\end{theorem}
This ring can be extended to the field $\mathcal{M}_{-1}$ of convolution quotients by following the lines of the classical Mikusi\'nski operational calculus:
\begin{equation*}
\mathcal{M}_{-1} = C_{-1} \times (C_{-1} \setminus \{0\}),
\end{equation*}
where the equivalence relation $(\sim)$ is defined, as usual, by
\begin{equation}
(f, g) \sim (f_1, g_1) \Leftrightarrow (f * g_1)(t) = (g * f_1)(t).
\end{equation}
For the sake of convenience, the elements of the field $\mathcal{M}_{-1}$ can be formally considered as convolution quotients $f / g$. The operations of addition and multiplication are then defined in $\mathcal{M}_{-1}$ as usual:
\begin{equation}\label{eq:addition}
\frac{f}{g} + \frac{f_1}{g_1} := \frac{f * g_1 + g * f_1}{g * g_1}
\end{equation}
and
\begin{equation}\label{eq:multiplication}
\frac{f}{g} \cdot \frac{f_1}{g_1} := \frac{f * f_1}{g * g_1}.
\end{equation}
\begin{theorem}
The space $\mathcal{M}_{-1}$ with the operations of addition \eqref{eq:addition} and multiplication \eqref{eq:multiplication} becomes a commutative field $(\mathcal{M}_{-1}, +, \cdot)$.
\end{theorem}
The ring $C_{-1}$ can be embedded into the field $\mathcal{M}_{-1}$ by the map ($\upalpha > 0$):
\begin{equation}\label{eq:7}
f \mapsto \frac{h_\upalpha * f}{h_\upalpha},
\end{equation}
with, $h_\upalpha(x) = x^{\upalpha-1}/\Gamma(\upalpha)$.

In the field $\mathcal{M}_{-1}$, the operation of multiplication with a scalar $\lambda$ from the field $\mathbb{R}$ (or $\mathbb{C}$) can be defined by the relation
\begin{equation}
\lambda \cdot \frac{f}{g} := \frac{\lambda f}{g}, \quad \frac{f}{g} \in \mathcal{M}_{-1}.
\end{equation}

For $\upalpha, \upbeta, \gamma, \delta \in \mathbb{C}$ with $\operatorname{Re} (\upalpha)>0$ and $\operatorname{Re}(\upbeta)>0$, we introduce the notation $\mathbb{P}_{\upbeta, \gamma} = \mathbb{P}_{\upalpha, \upbeta, \gamma, \delta} \in C_{-1}$. We often suppress the dependence $\upalpha$ and $\delta$ in the notation, simply writing $\mathbb{P}_{\upbeta, \gamma}$, because $\upbeta$ and $\gamma$ are the most important parameters for the algebraic structure. Namely, for fixed $\upalpha$ and $\delta$, by Lemma \ref{semigrouprop}, we have
\begin{equation*}
\mathbb{P}_{\upalpha, \upbeta, \gamma, \delta}^n:=\underbrace{\mathbb{P}_{\upalpha, \upbeta, \gamma, \delta} * \mathbb{P}_{\upalpha, \upbeta, \gamma, \delta} * \cdots * \mathbb{P}_{\upalpha, \upbeta, \gamma, \delta}}_{n \text { times }} = \mathbb{P}_{\upalpha, n \upbeta, n \gamma, \delta} .    
\end{equation*}
Thus, we can also define the fractional powers of this ring element, by
\begin{equation*}
 \mathbb{P}_{\upalpha, \upbeta, \gamma, \delta}^\kappa = \mathbb{P}_{\upalpha, \kappa \upbeta, \kappa \gamma, \delta}, \quad \kappa > 0,   
\end{equation*}
and so the set of all elements $\mathbb{P}_{\upbeta, \gamma} = \mathbb{P}_{\upalpha, \upbeta, \gamma, \delta}$ forms a multiplicative sub semigroup in $C_{-1}$. In the field $\mathcal{M}_{-1}$, this semigroup extends to a group, which is given by the set of all elements $\mathbb{P}_{\upbeta, \gamma}$ and their multiplicative inverses
\begin{equation*}
\mathbb{S}_{\upbeta, \gamma}:=\mathbb{P}_{\upbeta, \gamma}^{-1} \in \mathcal{M}_{-1}    
\end{equation*}
Now that we have negative powers as well as positive ones, we are able to define fractional powers to any real order of the elements $\mathbb{P}_{\upbeta, \gamma} \in C_{-1}$ and $\mathbb{S}_{\upbeta, \gamma} \in \mathcal{M}_{-1}$, namely as follows:
\begin{equation*}
\mathbb{P}_{\upbeta, \gamma}^\kappa= \begin{cases}\mathbb{P}_{\kappa \upbeta, \kappa \gamma} \in C_{-1} & \text { if } \kappa>0, \\ I \in \mathcal{M}_{-1} & \text { if } \kappa=0, \\ \mathbb{S}_{\kappa \upbeta, \kappa \gamma} \in \mathcal{M}_{-1} & \text { if } \kappa<0,\end{cases}   
\end{equation*}
and
\begin{equation*}
\mathbb{S}_{\upbeta, \gamma}^\kappa= \begin{cases}\mathbb{S}_{\kappa \upbeta, \kappa \gamma} \in \mathcal{M}_{-1} & \text { if } \kappa>0, \\ I \in \mathcal{M}_{-1} & \text { if } \kappa=0, \\ \mathbb{P}_{\kappa \upbeta, \kappa \gamma} \in C_{-1} & \text { if } \kappa<0,\end{cases}   
\end{equation*}
where $I$ denotes the multiplicative identity element in the field. The semigroup property of the set of elements $\mathbb{P}_{\upbeta, \gamma}$ also extends to a group property, so that we have
\begin{equation*}
\mathbb{P}_{\beta_1, \gamma_1} * \mathbb{S}_{\beta_2, \gamma_2}= \begin{cases}\mathbb{P}_{\beta_1-\beta_2, \gamma_1-\gamma_2} \in C_{-1} & \text { if } \beta_1>\beta_2, \\ I \in \mathcal{M}_{-1} & \text { if } \beta_1=\beta_2 \text { and } \gamma_1=\gamma_2, \\ \mathbb{S}_{\beta_2-\beta_1, \gamma_2-\gamma_1} \in \mathcal{M}_{-1} & \text { if } \beta_2>\beta_1,\end{cases}    
\end{equation*}
This study focuses on the singular case, and thus excludes the nonsingular class of Prabhakar operators (i.e., the case where $\beta_{1} = \beta_{2}$ and $\gamma_{1} \neq \gamma_{2}$). For a treatment of the nonsingular case and the properties of the elements $\mathbb{S}_{\upbeta, \gamma}$, we refer the reader to earlier works on Mikusiński's approach~\cite{rani2022mikusinski}. The term
\begin{align*}
\sum_{i,i_1,\ldots,i_n
=0}^{\infty}
\frac{(\gamma)_i\delta^i}{i!}\prod_{r=1}^{n}\frac{(-\gamma\theta_r)_{i_r}\delta^{i_r}}{i_r!} &\frac{x^{\upalpha i+\upalpha i_1+\cdots+\upalpha i_n+\upbeta+s_n-k-1}}
{\Gamma(\upalpha i+\upalpha i_1+\cdots+\upalpha i_n+\upbeta+s_n-k)}\\
&= 
E_{(\upalpha,\ldots,\upalpha),\,\upbeta+s_n-k}
^{(\gamma,-\gamma\theta_1,\ldots,-\gamma\theta_n)}
\left(\delta x^\upalpha,\ldots,\delta x^\upalpha\right)
\end{align*}
can be associated with the Prabhakar integral operator ${}_{0}\mathbb{E}_{x}^{\upalpha, \upbeta, \gamma, \delta}$ for some special parameters choices.
\begin{theorem}
Let $\upalpha>0$, let $\upbeta,s_k,\theta_k$ satisfy the conditions \eqref{peremeters}, and let $\gamma, \delta \in \mathbb{R}$. The $n$th level Prabhakar fractional differential operator ${ }_0 \mathbb{D}_{nL, x}^{\upalpha, \upbeta, \gamma, \delta; \theta_{n}}$ is represented in the field $\mathcal{M}_{-1}$ in the following form, for $f \in X_{nL}\cap C_{-1}$ such that ${ }_0 \mathbb{D}_{nL, x}^{\upalpha, \upbeta, \gamma, \delta; \theta_{n}}f\in C_{-1}$:
\begin{align}\label{algebricform}
{}_{0}\mathbb{D}^{\upalpha,\upbeta, \gamma, \delta; \theta_{n}}_{nL, x} f(x)  = \mathbb{S}_{\upbeta, \gamma} * f - \mathbb{S}_{\upbeta, \gamma} * \sum_{k=1}^{n}~
p_{k} x^{\upbeta+ \upsigma_{k}-k} E_{\upalpha,\,\upbeta+\upsigma_k-k+1}^{\gamma(1-\Theta_k)}\left(\delta x^{\upalpha}\right),
\end{align}
where the constants $p_{k}$ are given by
\begin{equation*}\label{constant}
p_{k} =\left[\left(\prod_{i=k+1}^{n}{}_{0}\mathbb{E}_{x}^{\upalpha,s_i,-\gamma\theta_i,\delta}\frac{d}{dx}\right){}_{0}\mathbb{E}_{x}^{\upalpha,n-\upbeta-\upsigma_n,-\gamma(1-\Theta_n),\delta}f\right](0),\qquad k=1,\dots,n .
\end{equation*}
\end{theorem}
\begin{proof}
Let $f \in X_{nL}\cap C_{-1}$ with ${ }_0 \mathbb{D}_{nL, x}^{\upalpha, \upbeta, \gamma, \delta; \theta_{n}}f\in C_{-1}$. Since ${}_{0}\mathbb{E}_{x}^{\upalpha,\upbeta,\gamma,\delta}g=\mathbb{P}_{\upbeta,\gamma}*g$, the projector Theorem~\ref{thm:proj} gives the following algebraic relation in the field $\mathcal{M}_{-1}$:
\begin{align*}
\mathbb{P}_{\upbeta, \gamma} *{}_{0} & \mathbb{D}^{\upalpha,\upbeta, \gamma, \delta;\theta_{n}}_{nL, x} f(x) 
= f(x) -\sum_{k=1}^{n}
p_{k}\, x^{\upbeta+ \upsigma_k-k} E_{\upalpha,\,\upbeta+\upsigma_k-k+1}^{\gamma(1-\Theta_k)}\left(\delta x^{\upalpha}\right).
\end{align*}
Here, the coefficients $p_k$ are constants. Upon multiplying both sides by the algebraic inverse $\mathbb{S}_{\upbeta, \gamma}$ of the function $\mathbb{P}_{\upbeta, \gamma}$, we get:
\begin{align*}
{}_{0}\mathbb{D}^{\upalpha,\upbeta, \gamma, \delta; \theta_{n}}_{nL, x} f(x)  = \mathbb{S}_{\upbeta, \gamma} * f - \mathbb{S}_{\upbeta, \gamma} *  \:\sum_{k=1}^{n}
p_{k}\cdot x^{\upbeta+ \upsigma_{k}-k} E_{\upalpha,\,\upbeta+\upsigma_k-k+1}^{\gamma(1-\Theta_k)}\left(\delta x^{\upalpha}\right).
\end{align*}
This is already in the form \eqref{algebricform}, after using the semigroup relation of the $\mathbb{P}$ and $\mathbb{S}$ elements of $\mathcal{M}_{-1}$, noting that the first parameter $\upbeta$ in $\mathbb{P}_{\upbeta, \gamma}$ and $\mathbb{S}_{\upbeta, \gamma}$ should be positive for clarity of the notation.

It remains to demonstrate that every term in \eqref{algebricform} is a well-defined element of $\mathcal{M}_{-1}$, that is, to show that the functions multiplying the constants belong to $C_{-1}$ and that the constants $p_k$ are finite numbers.

By the conditions \eqref{peremeters}, $\upbeta+\upsigma_k-k\le 0$ and, since $s_i<1$ for $i\ge2$ and $n-1<\upbeta+\upsigma_n$,
\begin{align*}
&\upbeta+\upsigma_k>n-1-(\upsigma_n-\upsigma_k)>n-1-(n-k)=k-1,\\
 &-1<\upbeta+\upsigma_k-k\le 0 .
\end{align*}
The Prabhakar function $E_{\upalpha,\,\upbeta+\upsigma_k-k+1}^{\gamma(1-\Theta_k)}(\delta x^{\upalpha})$ is a power series in $x^{\upalpha}$, $\upalpha>0$, hence it is continuous on $[0,\infty)$. Therefore
\begin{equation*}
x^{\upbeta+\upsigma_k-k}\,E_{\upalpha,\,\upbeta+\upsigma_k-k+1}^{\gamma(1-\Theta_k)}\left(\delta x^{\upalpha}\right)=x^{p}f_1(x),
\end{equation*}
where $p=\upbeta+\upsigma_k-k>-1$, $f_1\in C[0,\infty),$
which means that this function belongs to $C_{-1}$. Next, $p_k=h_{k+1}(0)$, where $h_{k+1}\in AC([0,X])$ for every $X>0$ by the definition of the space $X_{nL}$; hence $p_k$ is a well-defined finite number. Finally, $\mathbb{P}_{\upbeta,\gamma}\in C_{-1}$ because $\upbeta>0$, so its inverse $\mathbb{S}_{\upbeta,\gamma}$ exists in $\mathcal{M}_{-1}$, and the right-hand side of \eqref{algebricform} is an element of $\mathcal{M}_{-1}$ which, by the relation above, coincides with ${}_{0}\mathbb{D}^{\upalpha,\upbeta, \gamma, \delta; \theta_{n}}_{nL, x} f\in C_{-1}$. This establishes the required representation.
\end{proof} 
\section{Application}\label{S5}
In this section, we present some applications of the results presented in the preceding sections by solving fractional differential and integral equations involving $n$th level Prabhakar derivatives, employing the algebraic Mikusiński method. 

Research on equations involving Prabhakar-type fractional operators has expanded in several directions. For instance, Karimov and Toshpulatov \cite{karimov2025mixed} addressed a mixed wave diffusion model governed by a Prabhakar derivative. In another study, Karimov and Turdiev \cite{KarimovTurdiev2023} investigated a subdiffusion equation subject to a Wentzell Neumann boundary condition in the setting of the Hilfer Prabhakar derivative. Turdiev \cite{Turdiev2025} considered a nonlocal formulation for a time fractional hyperbolic type equation incorporating the Prabhakar fractional derivative. Additional related developments include the work of Bulavatsky \cite{bulavatsky2017mathematical}, who modeled fractional filtration dynamics using Hilfer-Prabhakar-type operators. These contributions demonstrate that many physically relevant models lead to fractional differential equations involving Prabhakar-type operators. In this section, we use the results established earlier to provide a unified framework for solving fractional differential and integral equations arising in these models.
\subsection{Generalized Fractional Integro-Differential Equations}
We consider the following initial value problem involving the $n$th level Prabhakar fractional derivative.
\begin{equation}\label{ex1}
\left\{
\begin{array}{ll}
{}_{0}\mathbb{D}_{nL,x}^{\upalpha,\upbeta,\gamma,\delta;\theta_n}
\,y(x)-\lambda\,y(x)=f(x), & x>0, \\[6pt]
h_{k+1}(0)=a_k, & k=1,\ldots,n-1.
\end{array}
\right.
\end{equation}
where $y\in X_{nL}$, $0<\upbeta\le1$, $n-1<\upbeta+\upsigma_n$, $\theta_k\ge0$, $\Theta_n\le1$, $\lambda,a_k\in\mathbb{C}$, and $f\in L^1_{\mathrm{Loc}}$. We use the elements of the Mikusi\'nski field $\mathcal{M}_{-1}$,
\begin{equation*}
\mathbb{P}_{\upbeta,\gamma}(x)=x^{\upbeta-1}E^{\gamma}_{\upalpha,\upbeta}\left(\delta x^{\upalpha}\right)\ (\gamma>0),\qquad \mathbb{S}_{\upbeta,\gamma}=\mathbb{P}_{\upbeta,\gamma}^{-1},
\end{equation*}
with $\mathbb{P}_{\upbeta_1,\gamma_1}*\mathbb{P}_{\upbeta_2,\gamma_2}=\mathbb{P}_{\upbeta_1+\upbeta_2,\gamma_1+\gamma_2}$, and ${}_{0}\mathbb{E}_{x}^{\upalpha,\upbeta,\gamma,\delta}u=\mathbb{P}_{\upbeta,\gamma}*u$. The functions of the kernel of the derivative are
\begin{equation*}
x^{\upbeta+\upsigma_k-k}\,E^{\gamma(1-\Theta_k)}_{\upalpha,\,\upbeta+\upsigma_k-k+1}\left(\delta x^{\upalpha}\right)=\mathbb{P}_{\upbeta+\upsigma_k-k+1,\ \gamma(1-\Theta_k)}(x),\qquad k=1,\dots,n,
\end{equation*}
and they belong to $C_{-1}$, because $\upbeta+\upsigma_k-k>-1$.

By the projector Theorem~\ref{thm:proj} and $h_{k+1}(0)=a_k$, now by writing the  initial value problem in the algebraic form using equation \eqref{algebricform}
\begin{align}
\mathbb{S}_{\upbeta, \gamma} * y - \mathbb{S}_{\upbeta, \gamma} * \sum_{k=1}^{n}~
p_{k} x^{\upbeta+ \upsigma_{k}-k} E_{\upalpha,\,\upbeta+\upsigma_k-k+1}^{\gamma(1-\Theta_k)}\left(\delta x^{\upalpha}\right) -\lambda y(x) = f(x),
\end{align}
Rearranging the algebraic equation gives
\begin{align*}
 &\mathbb{S}_{\upbeta,\gamma}*y-\mathbb{S}_{\upbeta,\gamma}*\sum_{k=1}^{n}a_k\,\mathbb{P}_{\upbeta+\upsigma_k-k+1,\ \gamma(1-\Theta_k)}-\lambda\,y =f \\
 &\left(\mathbb{S}_{\upbeta,\gamma}-\lambda I\right)*y =f+\mathbb{S}_{\upbeta,\gamma}*\sum_{k=1}^{n}a_k\,\mathbb{P}_{\upbeta+\upsigma_k-k+1,\ \gamma(1-\Theta_k)}
\end{align*}
where $I$ is the multiplicative identity element of the field $\mathcal{M}_{-1}$. Consequently, we have
\begin{align*}
y &=\frac{f}{\mathbb{S}_{\upbeta,\gamma}-\lambda I}+\sum_{k=1}^{n}a_k\,\frac{\mathbb{S}_{\upbeta,\gamma}*\mathbb{P}_{\upbeta+\upsigma_k-k+1,\ \gamma(1-\Theta_k)}}{\mathbb{S}_{\upbeta,\gamma}-\lambda I}\\
&=\frac{f*\mathbb{P}_{\upbeta,\gamma}}{I-\lambda\mathbb{P}_{\upbeta,\gamma}}+\sum_{k=1}^{n}a_k\,\frac{\mathbb{P}_{\upbeta+\upsigma_k-k+1,\ \gamma(1-\Theta_k)}}{I-\lambda\mathbb{P}_{\upbeta,\gamma}}\\
&=f*\mathbb{P}_{\upbeta,\gamma}*\left(I-\lambda\mathbb{P}_{\upbeta,\gamma}\right)^{-1}+\sum_{k=1}^{n}a_k\,\mathbb{P}_{\upbeta+\upsigma_k-k+1,\ \gamma(1-\Theta_k)}*\left(I-\lambda\mathbb{P}_{\upbeta,\gamma}\right)^{-1},
\end{align*}
where we used $\frac{\mathbb{S}_{\upbeta,\gamma}}{\mathbb{S}_{\upbeta,\gamma}-\lambda I}=\left(I-\lambda\mathbb{P}_{\upbeta,\gamma}\right)^{-1}$. At this stage, all the expressions are represented by functions of $C_{-1}$. As shown in [\cite{fernandez2023operational}, Example 3.9], and by the semigroup property $\mathbb{P}_{\upbeta,\gamma}^{\,i}=\mathbb{P}_{i\upbeta,i\gamma}$,
\begin{equation*}
\left(I-\lambda\mathbb{P}_{\upbeta,\gamma}\right)^{-1}=\sum_{i=0}^{\infty}\lambda^{i}\,\mathbb{P}_{i\upbeta,\,i\gamma},\qquad \mathbb{P}_{0,0}=I .
\end{equation*}
\begin{equation*}
y=\sum_{i=0}^{\infty}\lambda^{i}\,f*\mathbb{P}_{(i+1)\upbeta,\,(i+1)\gamma}+\sum_{k=1}^{n}a_k\sum_{i=0}^{\infty}\lambda^{i}\,\mathbb{P}_{\upbeta+\upsigma_k-k+1,\ \gamma(1-\Theta_k)}*\mathbb{P}_{i\upbeta,\,i\gamma},
\end{equation*}
and by the index law $\mathbb{P}_{\upbeta+\upsigma_k-k+1,\ \gamma(1-\Theta_k)}*\mathbb{P}_{i\upbeta,i\gamma}=\mathbb{P}_{(i+1)\upbeta+\upsigma_k-k+1,\ \gamma(i+1-\Theta_k)}$, for $i=0$ this is the identity $\mathbb{P}_{\mu,\rho}*I=\mathbb{P}_{\mu,\rho}$.

Writing the convolution as an integral and $\mathbb{P}_{\upbeta,\gamma}(x)=x^{\upbeta-1}E^{\gamma}_{\upalpha,\upbeta}(\delta x^{\upalpha})$, we obtain this:
\begin{align}\label{biveriatemlfun}
y(x)=&\sum_{i=0}^{\infty}\lambda^{i}\int_0^x(x-\eta)^{(i+1)\upbeta-1}E^{(i+1)\gamma}_{\upalpha,\,(i+1)\upbeta}\left(\delta(x-\eta)^{\upalpha}\right)f(\eta)\,\mathrm{d}\eta\nonumber\\
&+\sum_{k=1}^{n}a_k\sum_{i=0}^{\infty}\lambda^{i}\,x^{(i+1)\upbeta+\upsigma_k-k}\,E^{\gamma(i+1-\Theta_k)}_{\upalpha,\,(i+1)\upbeta+\upsigma_k-k+1}\left(\delta x^{\upalpha}\right).
\end{align}
This is the solution of \eqref{ex1}

Expanding the Prabhakar functions, the kernel of the first term is
\begin{align*}
\sum_{i=0}^{\infty}\lambda^{i}&(x-\eta)^{(i+1)\upbeta-1}\sum_{j=0}^{\infty}\frac{(\gamma+\gamma i)_{j}\left(\delta(x-\eta)^{\upalpha}\right)^{j}}{j!\,\Gamma(\upalpha j+i\upbeta+\upbeta)}\\
&=\sum_{i=0}^{\infty}\sum_{j=0}^{\infty}(x-\eta)^{\upbeta-1}\,\frac{\Gamma(\gamma+\gamma i+j)}{j!\,\Gamma(\gamma+\gamma i)}\,\frac{\left(\lambda(x-\eta)^{\upbeta}\right)^{i}\left(\delta(x-\eta)^{\upalpha}\right)^{j}}{\Gamma(\upalpha j+\upbeta i+\upbeta)} ,
\end{align*}
and, for $\gamma>0$ and $\gamma(1-\Theta_k)>0$, 
\begin{align*}
\sum_{i=0}^{\infty}\lambda^{i}\,& x^{\upbeta+\upsigma_k-k+\upbeta i}\sum_{j=0}^{\infty}\frac{(\gamma(1-\Theta_k)+\gamma i)_{j}\left(\delta x^{\upalpha}\right)^{j}}{j!\,\Gamma(\upalpha j+\upbeta i+\upbeta+\upsigma_k-k+1)}\\
&=\sum_{i=0}^{\infty}\sum_{j=0}^{\infty}x^{\upbeta+\upsigma_k-k}\,\frac{\Gamma(\gamma(1-\Theta_k)+\gamma i+j)}{j!\,\Gamma(\gamma(1-\Theta_k)+\gamma i)}\,\frac{\left(\lambda x^{\upbeta}\right)^{i}\left(\delta x^{\upalpha}\right)^{j}}{\Gamma(\upalpha j+\upbeta i+\upbeta+\upsigma_k-k+1)}.
\end{align*}
Recall the $E_{2}$ function
\begin{align*}
E_{2}\begin{pmatrix}\gamma_{1},\alpha_{1},\beta_{1};\gamma_{2},\alpha_{2}&\big|~x\\ \delta_{1},\alpha_{3},\beta_{2};\delta_{2},\alpha_{4};\delta_{3},\beta_{3}&\big|~y\end{pmatrix}&=\sum_{m=0}^{\infty}\sum_{p=0}^{\infty}\frac{(\gamma_{1})_{\alpha_{1}m+\beta_{1}p}(\gamma_{2})_{\alpha_{2}m}}{\Gamma(\delta_{1}+\alpha_{3}m+\beta_{2}p)}\\
&\qquad\times\frac{x^{m}}{\Gamma(\delta_{2}+\alpha_{4}m)}\,\frac{y^{p}}{\Gamma(\delta_{3}+\beta_{3}p)},
\end{align*}
where $\gamma_1,\gamma_2,\delta_1,\delta_2,\delta_3,x,y\in\mathbb{C}$
and $\min\{\upalpha_1,\upalpha_2,\upalpha_3,\upalpha_4,
\upbeta_1,\upbeta_2,\upbeta_3\}>0. $  
For further details, we refer the reader to~\cite{garg2013mittag}.

Motivated by the recent development of several bivariate
generalizations of the classical Mittag-Leffler function,
we reformulate equation~\eqref{biveriatemlfun} into a more convenient
alternative form. In particular, we obtain the following
representations:
\begin{equation*}
\left\{
\begin{aligned}
\gamma_1 &= \gamma,
& \alpha_1 &= \gamma,
& \beta_1 &= 1,\\
\gamma_2 &= 1,
& \alpha_2 &= 0,\\
\delta_1 &= \upbeta,
& \alpha_3 &= \upbeta,
& \beta_2 &= \upalpha,\\
\delta_2 &= \gamma,
& \alpha_4 &= \gamma,\\
\delta_3 &= 1,
& \beta_3 &= 1.
\end{aligned}
\right.
\end{equation*}
Similarly, for the boundary series, the corresponding
parameters are specified as
\begin{equation*}
\left\{
\begin{aligned}
\gamma_1 &= \gamma(1-\Theta_k),
& \alpha_1 &= \gamma,
& \beta_1 &= 1,\\
\gamma_2 &= 1,
& \alpha_2 &= 0,\\
\delta_1 &= \upbeta+\upsigma_k-k+1,
& \alpha_3 &= \upbeta,
& \beta_2 &= \upalpha,\\
\delta_2 &= \gamma(1-\Theta_k),
& \alpha_4 &= \gamma,\\
\delta_3 &= 1,
& \beta_3 &= 1.
\end{aligned}
\right.
\end{equation*}
Therefore, for $\gamma>0$ and $\Theta_k<1$ $(k=1,\dots,n)$, the solution \eqref{biveriatemlfun} takes the form
\begin{align}\label{solE2}
y(x)=&\ \Gamma(\gamma)\int_0^x(x-\eta)^{\upbeta-1}E_{2}\left(\begin{matrix}\gamma,\ \gamma,\ 1,\ 1,\ 0\\ \upbeta,\ \upbeta,\ \upalpha,\ \gamma,\ \gamma,\ 1,\ 1\end{matrix}\ \middle|\ \begin{matrix}\lambda(x-\eta)^{\upbeta}\\ \delta(x-\eta)^{\upalpha}\end{matrix}\right)f(\eta)\,\mathrm{d}\eta\nonumber\\
& ~~+\sum_{k=1}^{n}a_k\,\Gamma\left(\gamma(1-\Theta_k)\right)x^{\upbeta+\upsigma_k-k}\, \times \nonumber \\
&~~E_{2}\left(\begin{matrix}\gamma(1-\Theta_k),\ \gamma,\ 1,\ 1,\ 0\\ \upbeta+\upsigma_k-k+1,\ \upbeta,\ \upalpha,\ \gamma(1-\Theta_k),\ \gamma,\ 1,\ 1\end{matrix}\ \middle|\ \begin{matrix}\lambda x^{\upbeta}\\ \delta x^{\upalpha}\end{matrix}\right).
\end{align}
\begin{remark}
For particular choices of the parameters, we obtain the following special cases:
\begin{enumerate}
\item[(a)] For $\lambda=0$ the solution is $y={}_{0}\mathbb{E}_{x}^{\upalpha,\upbeta,\gamma,\delta}f+\sum_ka_k\mathbb{P}_{\upbeta+\upsigma_k-k+1,\gamma(1-\Theta_k)}$, in agreement with the projector theorem. For $\gamma=\delta=0$, $n=1$, $s_1=0$ the Riemann--Liouville case \eqref{biveriatemlfun} becomes
\begin{equation*}
y(x)=\int_0^x(x-\eta)^{\upbeta-1}E_{\upbeta,\upbeta}\left(\lambda(x-\eta)^{\upbeta}\right)f(\eta)\,\mathrm{d}\eta+a_1x^{\upbeta-1}E_{\upbeta,\upbeta}\left(\lambda x^{\upbeta}\right),    
\end{equation*} 
the classical solution.

\item[(b)]For $s_1=\theta_1(1-\upbeta)$ the boundary function is
\begin{equation*}
x^{(i+1)\upbeta+\theta_1(1-\upbeta)-1}\,E^{\gamma(i+1-\theta_1)}_{\upalpha,\,(i+1)\upbeta+\theta_1(1-\upbeta)}\left(\delta x^{\upalpha}\right),
\end{equation*}
and \eqref{biveriatemlfun} is the solution of the Hilfer--Prabhakar equation with the initial condition $h_2(0)=a_1$. 
\end{enumerate}
\end{remark}
\subsection{Time fractional heat equation}
In recent years, considerable attention has been devoted to the study of time fractional heat equations, reflecting their important role in mathematical physics, probability theory, and anomalous diffusion modeling; see, for example, \cite{oldham2010fractional,sokolov2004fractional,tomovski2020applications}.
Motivated by these developments, the present work examines a generalized form of the time fractional heat equation involving $n$th-level Prabhakar fractional derivatives. Within this framework, we derive new analytical results that extend existing approaches and contribute to the broader understanding of fractional diffusion processes \cite{garrappa2016fractional,giusti2020practical}.

Consider the time-fractional diffusion equation
\begin{equation}\label{pde_corrected}
{}_{0}\mathbb{D}_{nL,t}^{\upalpha,\upbeta,\gamma,\delta;\theta_{n}}u(x,t)=\tilde{k}\,\frac{\partial^{2}}{\partial x^{2}}u(x,t),\qquad -\infty<x<\infty,\ t>0,\quad\tilde{k}>0,
\end{equation}
with $0<\upbeta\le1$, $n-1<\upbeta+\upsigma_n$, $\theta_k\ge0$, $\Theta_n\le1$, the initial conditions
\begin{equation}\label{pde_ic}
h_{k+1}(x,0)=A_k(x),\qquad k=1,\dots,n,\quad A_k\in L^1(\mathbb{R}),
\end{equation}
and the boundary conditions
\begin{equation}\label{pde_bc}
u(x,t)\to0,\quad\frac{\partial u}{\partial x}(x,t)\to0\quad\text{as }x\to\pm\infty,\ t>0 ,
\end{equation}
where $u(x,\cdot)\in X_{nL}$ and $h_{k+1}(x,\cdot)$ are the functions from the definition of $X_{nL}$ for $u(x,\cdot)$, so that $h_{k+1}(x,0)$ are the coefficients $p_k$ of Theorem~\ref{thm:proj}. We write
\begin{equation*}
\mathbb{P}_{\mu,\rho}(t)=t^{\mu-1}E^{\rho}_{\upalpha,\mu}\left(\delta t^{\upalpha}\right),
\end{equation*}

Let $\hat{u}(\omega,t)=\int_{-\infty}^{\infty}u(x,t)\,e^{\iota\omega x}\,\mathrm{d}x$ and $\hat{A}_k=\mathcal{F}[A_k]$. By \eqref{pde_bc}, $\mathcal{F}\left[\partial_x^{2}u\right]=-\omega^{2}\hat{u}$. using the Fourier transform on the operator  ${}_{0}\mathbb{D}_{nL,t}^{\upalpha,\upbeta,\gamma,\delta;\theta_{n}}$ and the functions $h_{k+1}$ we have 
\begin{align}\label{fode}
&{}_{0}\mathbb{D}_{nL,t}^{\upalpha,\upbeta,\gamma,\delta;\theta_{n}}\hat{u}(\omega,t)=-\tilde{k}\,\omega^{2}\,\hat{u}(\omega,t), \nonumber\\
&h_{k+1}(\omega,0)=\hat{A}_k(\omega),\quad k=1,\dots,n .
\end{align}
By \eqref{algebricform}  and the algebraic form of the derivative, from Theorem~\ref{thm:proj}, problem \eqref{fode} reads
\begin{equation*}
\mathbb{S}_{\upbeta,\gamma}*\hat{u}-\mathbb{S}_{\upbeta,\gamma}*\sum_{k=1}^{n}\hat{A}_k(\omega)\,\mathbb{P}_{\upbeta+\upsigma_k-k+1,\gamma(1-\Theta_k)}=-\tilde{k}\,\omega^{2}\,\hat{u},
\end{equation*}
hence $\left(\mathbb{S}_{\upbeta,\gamma}+\tilde{k}\omega^{2}I\right)*\hat{u}=\mathbb{S}_{\upbeta,\gamma}*\sum_k\hat{A}_k(\omega)\mathbb{P}_{\upbeta+\upsigma_k-k+1,\gamma(1-\Theta_k)}$. Multiplying by $\mathbb{P}_{\upbeta,\gamma}$ and dividing by $\mathbb{S}_{\upbeta,\gamma}+\tilde{k}\omega^{2}I$. We use $\mathbb{S}_{\upbeta,\gamma}*\mathbb{P}_{\upbeta,\gamma}=I$ and the field structure of $\mathcal{M}_{-1}$,
\begin{equation*}
\hat{u}(\omega,\cdot)=\sum_{k=1}^{n}\hat{A}_k(\omega)\,\mathbb{P}_{\upbeta+\upsigma_k-k+1,\gamma(1-\Theta_k)}*\left(I+\tilde{k}\,\omega^{2}\,\mathbb{P}_{\upbeta,\gamma}\right)^{-1}.
\end{equation*}
As in the previous application, $\left(I+\tilde{k}\omega^{2}\mathbb{P}_{\upbeta,\gamma}\right)^{-1}=\sum_{i=0}^{\infty}\left(-\tilde{k}\omega^{2}\right)^{i}\mathbb{P}_{i\upbeta,i\gamma}$ for every $\omega$, and the index law gives
\begin{align}\label{fourier_sol}
\hat{u}(\omega,t) = \sum_{k=1}^{n}\hat{A}_k(\omega)\,\sum_{i=0}^{\infty}\left(-\tilde{k}\,\omega^{2}\right)^{i}t^{(i+1)\upbeta+\upsigma_k-k}\,E^{\gamma(i+1-\Theta_k)}_{\upalpha,\,(i+1)\upbeta+\upsigma_k-k+1}\left(\delta t^{\upalpha}\right).
\end{align}
Applying the inverse Fourier transform and performing some algebraic simplifications, we obtain the following expression.
Define, for $k=1,\dots,n$, the kernel
\begin{equation}\label{Ghat}
\hat{G}_k(\omega,t)=\sum_{i=0}^{\infty}\left(-\tilde{k}\,\omega^{2}\right)^{i}
t^{(i+1)\upbeta+\upsigma_k-k}\,
E^{\gamma(i+1-\Theta_k)}_{\upalpha,\,(i+1)\upbeta+\upsigma_k-k+1}\!\left(\delta t^{\upalpha}\right),
\end{equation}
so that \eqref{fourier_sol} reads
\begin{equation}
\hat{u}(\omega,t)=\sum_{k=1}^{n}\hat{A}_k(\omega)\,\hat{G}_k(\omega,t).
\end{equation}
Let $G_k(x,t)$ be the inverse Fourier transform of $\hat{G}_k(\cdot,t)$,
\begin{equation}\label{Gk}
G_k(x,t)=\frac{1}{2\pi}\int_{-\infty}^{\infty}\hat{G}_k(\omega,t)\,e^{-\iota\omega x}\,\mathrm{d}\omega .
\end{equation}
Since the inverse Fourier transform of a product is a convolution,
\begin{equation}
u(x,t)=\sum_{k=1}^{n}\int_{-\infty}^{\infty}G_k(x-\xi,t)\,A_k(\xi)\,\mathrm{d}\xi ,
\end{equation}
which is the solution of the PDE \eqref{pde_corrected}.

\begin{remark}
For particular choices of the parameters, the following special cases arise.
\begin{itemize}
\item[(a)] Formally, $\mathcal{F}^{-1}\left[\omega^{2i}\hat{A}_k\right]=(-1)^{i}A_k^{(2i)}$, and term-by-term inversion of \eqref{fourier_sol} gives
\begin{equation*}
u(x,t)=\sum_{k=1}^{n}\sum_{i=0}^{\infty}\tilde{k}^{\,i}\,A_k^{(2i)}(x)\,t^{(i+1)\upbeta+\upsigma_k-k}\,E^{\gamma(i+1-\Theta_k)}_{\upalpha,\,(i+1)\upbeta+\upsigma_k-k+1}\left(\delta t^{\upalpha}\right),
\end{equation*}
which is only a formal series for $A_k=\delta$, the Dirac function.

\item[(b)] For $n=1$ and $s_1=\theta_1(1-\upbeta)$, problem \eqref{fode} is the Hilfer--Prabhakar equation with the initial condition $h_2(\omega,0)=\hat{A}_1(\omega)$, and \eqref{pde_corrected} is the fractional heat equation governed by the non-regularized Hilfer--Prabhakar derivative, as in \textup{\cite{garra2014hilfer}}. For $\gamma=\delta=0$, $n=1$, $s_1=0$ we get $\hat{G}_k(\omega,t) =t^{\upbeta-1}E_{\upbeta,\upbeta}\left(-\tilde{k}\omega^{2}t^{\upbeta}\right)$, the classical solution of the time-fractional diffusion equation with the Riemann--Liouville derivative.
\end{itemize}
\end{remark}
\section{Conclusion}
In this work, we present the $n$th-level Prabhakar fractional derivative as a generalization of some classical fractional derivatives. Fundamental properties of this operator were established, with particular emphasis on its relationship with the corresponding Prabhakar fractional integral. By formulating a Mikusiński-type operational calculus for the $n$th-level Prabhakar framework, we provided an approach for handling fractional differential equations involving this multi-parameter operator. This operational perspective not only simplifies the analytical treatment of such equations but also highlights the structural connections between different fractional models within a single unified setting.

To demonstrate the applicability of the proposed theory, we derived explicit analytical solutions for a fractional-order ordinary differential equation and a time-fractional heat equation involving the $n$th-level Prabhakar derivative. These examples illustrate the adaptability of the developed operational calculus and its potential for modeling complex dynamical processes with memory and hereditary effects. The results presented in this paper open the way for further investigations into more general fractional partial differential equations, numerical methods, and applications in physics and engineering, where the flexibility of the $n$th-level Prabhakar derivative may provide more accurate and adaptable models for real-world phenomena.

\section*{Acknowledgments}
Imtiaz Waheed gratefully acknowledges the Higher Education Commission (HEC) of Pakistan for its support through the IRSIP program and expresses sincere thanks to Prof. Michael Ruzhansky and the Department of Mathematics: Analysis, Logic, and Discrete Mathematics at Ghent University for their valuable guidance and assistance.

\end{document}